\documentclass{amsart}

\usepackage{amsmath,amssymb,amsfonts,amsthm,latexsym}
\usepackage{enumerate}

\newtheorem{theorem}{Theorem}[section]
\newtheorem{lemma}[theorem]{Lemma}
\newtheorem{proposition}[theorem]{Proposition}
\newtheorem{corollary}[theorem]{Corollary}

\theoremstyle{definition}

\theoremstyle{remark}
\newtheorem{remark}[theorem]{Remark}

\newcommand{\C}{\mathbb{C}}
\newcommand{\Sp}{\mathbb{S}}
\newcommand{\Q}{\mathbb{Q}}

\newcommand{\Z}{\mathbb{Z}}
\newcommand{\N}{\mathbb{N}}

\newcommand{\SU}{\operatorname{SU}}
\newcommand{\SL}{\operatorname{SL}}

\newcommand{\lcm}{\operatorname{lcm}}

\newcommand{\co}{\colon\thinspace}

\newcommand{\bs}{\boldsymbol}

\newcommand{\Hilb}{\operatorname{Hilb}}

\begin{document}

\title{The Hilbert series of $\SL_2$-invariants}

\author[P.~de Carvalho Cayres Pinto]{Pedro de Carvalho Cayres Pinto}
\address{PEE/COPPE, Universidade Federal do Rio de Janeiro,
Av. Athos da Silveira Ramos 149, Centro de Tecnologia - Bloco H, Caixa postal 68504, CEP: 21941-972, Rio de Janeiro, Brazil}
\email{pedrocayres@poli.ufrj.br}

\author[H.-C.~Herbig]{Hans-Christian Herbig}
\address{Departamento de Matem\'{a}tica Aplicada,
Av. Athos da Silveira Ramos 149, Centro de Tecnologia - Bloco C, CEP: 21941-909 - Rio de Janeiro, Brazil}
\email{herbighc@gmail.com}

\author[D.~Herden]{Daniel Herden}
\address{Department of Mathematics, Baylor University,
One Bear Place \#97328,
Waco, TX 76798-7328, USA}
\email{Daniel\_Herden@baylor.edu}

\author[C.~Seaton]{Christopher Seaton}
\address{Department of Mathematics and Computer Science,
Rhodes College, 2000 N. Parkway, Memphis, TN 38112}
\email{seatonc@rhodes.edu}

\thanks{C.S., D.H., and H.-C.H. were supported by a Collaborate@ICERM grant from the Institute for Computational
and Experimental Research in Mathematics (ICERM). C.S. was supported by the E.C.~Ellett Professorship in Mathematics
and the Instituto de Matem\'{a}tica Pura e Aplicada (IMPA);  H.-C.H. was supported by CNPq through the
\emph{Plataforma Integrada Carlos Chagas.}}


\begin{abstract}
Let $V$ be a finite dimensional representations of the group $\SL_2$ of $2\times 2$ matrices with complex
coefficients and determinant one. Let $R=\C[V]^{\SL_2}$ be the algebra of $\SL_2$-invariant polynomials
on $V$. We present a calculation of the Hilbert series $\Hilb_R(t)=\sum_{n\ge 0}\dim (R_n)\: t^n$ as well
as formulas for the first four coefficients of the Laurent expansion of $\Hilb_R(t)$ at $t=1$.
\end{abstract}

\maketitle
\tableofcontents


\section{Introduction}
\label{sec:Intro}

Let $V$ be a finite dimensional representation of $\SL_2$. It is well-known that $V$ is isomorphic to a sum of
irreducible representations $\bigoplus_{k=1}^r V_{d_k}$. Here, $V_{d_k}$ stands for the $(d_k+1)$-dimensional
irreducible  representation of $\SL_2$ which is given by binary forms of degree $d_k\in \N$. In the decomposition $\bigoplus_{k=1}^r V_{d_k}$ it is not assumed that  the $d_k$ are pairwise distinct. The algebra of polynomial
$\SL_2$-invariants $R:=\C[V]^{\SL_2}$ is a finitely generated $\C$-algebra and carries a natural $\N$-grading
$R=\bigoplus_{n\ge 0} R_n$. In fact it is generated by a complete system of homogeneous invariants which obey some
homogeneous relations. For a more detailed discussion, the reader may consult \cite{DerskenKemperBook,PopovVinberg}.

In this paper we study the \emph{Hilbert series} $\Hilb_R(t)$ of $R$, i.e. the generating function that counts the dimensions of the homogeneous components $R_n$:
\begin{align*}
\Hilb_R(t)= \sum_{n=0}^\infty \dim(R_n)\: t^n.
\end{align*}
It is a classical result that $\Hilb_R(t)$ is rational. The degree of $\Hilb_R(t)$, i.e. the difference of the degree of the numerator and the degree of the denominator, is referred to as the \emph{$a$-invariant} $a(R)$ of $R$, see
\cite[Definitions 3.6.13 and 4.4.4]{BrunsHerzog} and \cite[Section 3]{GotoWatanabe}. The $a$-invariant of invariant
rings has been studied for example in \cite{CowieHerbigSeatonHerden,Knop,KnopLittelmann,PopovBook}; note that some references use $q$ to denote the negative $a$-invariant.

It is well-known that $\Hilb_R(t)$ has a pole at $t = 1$ of order equal the Krull dimension $\dim R$.
We use the notation $\gamma_m$, $m\ge 0$, to denote the coefficients  in the Laurent expansion
\begin{equation}
\label{eq:DefGammas}
    \Hilb_R(t) = \sum\limits_{m=0}^\infty \frac{\gamma_m}{(1 - t)^{\dim R-m}}.
\end{equation}
Some authors systematically use the notation $\gamma$ or $\deg(R)$ for $\gamma_0$ and $\tau$ or $\psi(R)$ for $\gamma_1$.
The coefficients $\gamma_0$ and $\gamma_1$ have clear interpretations in
the case of invariants of a finite group, see \cite[Lemma 2.4.4]{SturmfelsBook} as well as \cite[Section 3.13]{BensonBook}
or \cite{BensonCrawleyBoevey}, and their meaning in more general contexts has been investigated, e.g. in
\cite{AvramovBuchweitzSally} and \cite[Chapter~3]{PopovBook}.
For the case of invariants of $\SL_2$, David Hilbert \cite{Hilbert} published in 1893 a formula for $\gamma_0$ in the
case that $V=V_d$ is irreducible for $d\ge 5$:
\begin{equation} \label{eq:Hilbert}
    \gamma_0
        =   \frac{-1}{(3-(-1)^d)\:d!}\sum_{n=0}^{\lfloor d/2\rfloor}{d\choose n}
                \left(\frac{d}{2} - n \right)^{d-3}\quad.
\end{equation}
Since then, computations of $\Hilb_R(t)$ and $\gamma_0$ have been taken up by numerous authors, e.g.
\cite{BedratyukPoincareCovariants,BedratyukBivarPoincare,BedratyukSL2Invariants,BedratyukSL2MiltigradPoincare,
BedratyukBedratyukMultivarPoincare, BedratyukIlashCovariants,IlashPoincareNLinForm,
IlashPoincareNQuadForm,LittelmannProcesi,Springer}.

The main result of this paper is  Theorem~\ref{thrm:Main}, presenting formulas for $\gamma_0$, $\gamma_1$, $\gamma_2$,
and $\gamma_3$ that generalize Hilbert's formula in Equation~\eqref{eq:Hilbert} to encompass all finite-dimensional
representations of $\SL_2$ up to a set of low dimensional exceptions that can easily be computed directly.
As a corollary, we reproduce the computation
of the $a$-invariant given by F. Knop and P. Littelmann \cite{KnopLittelmann}. Note that
given our computation of $\gamma_0$, Knop and Littelmann's computation of the $a$-invariant
renders our computation of $\gamma_1$ superfluous, see
Remark~\ref{rem:KnopLittelmanGam1Superfluous}; we include this computation for readability towards
the computation of $\gamma_2$
and as an alternate derivation of the $a$-invariant. On the way, we produce
formulas in Proposition~\ref{prop:MultivarHilbSer} and Theorem~\ref{thrm:UnivarHilbSer} for
the multivariate and univariate Hilbert series of $\C[V]^{\SL_2}$, respectively, that in
particular indicate an algorithm for computing the Hilbert
series. This algorithm is described in Section~\ref{sec:Algorithm}; it has been implemented using \emph{Mathematica}
\cite{Mathematica} and is available from the authors by request.
Note that formulas for the multivariate Hilbert series have been given previously
by Brion \cite[Th\'{e}or\`{e}me 1]{Brion} and Bedratyuk--Bedratyuk in
\cite[Theorem 3]{BedratyukBedratyukMultivarPoincare}.

This paper is the third in a series that uses the methods described in
\cite[Section~4.6.1 and 4.6.4]{DerskenKemperBook} to systematically compute Hilbert series
of rings of invariants and give explicit, general descriptions for the first few Laurent coefficients $\gamma_m$.
The techniques were first laid out in \cite{HerbigSeaton}, where the Hilbert series of algebras of
regular functions on linear symplectic circle quotients were investigated. As explained in that reference,
that computation is equivalent to the computation corresponding to the invariant ring of a \emph{cotangent-lifted} representation of the circle, and the extension of these techniques to arbitrary circle representations was recently
presented in \cite{CowieHerbigSeatonHerden}. The key observation is related to weights of the Cartan torus that occur
with multiplicity in the representation, which in the case considered here occurs whenever $V$ contains two irreducible
representations $V_{d_k}$ whose dimensions have the same parity (and hence must occur whenever $r > 2$, i.e. $V$ has more
than two irreducible summands). Though these
degeneracies impose difficulties in the computation of the Hilbert series, they can be circumvented by taking
advantage of certain analytic continuations, viewing some instances of the integer weights as real parameters and
perturbing them to avoid degeneracies. This in particular can be used to show that the bare expressions
for the $\gamma_m$ in terms of the weights have removable singularities along the diagonals. After removing these
singularities, the resulting expressions can be expressed in terms of Schur polynomials, yielding succinct,
general expressions for the first four $\gamma_m$.
The potential usefulness of our technique to representations of reductive Lie groups of higher rank is currently
being explored, in particular for the case of a torus of dimension $\ell>1$. In a forthcoming paper, we will present
the computation of the first Laurent coefficient of the Hilbert series of the algebra of real regular functions on
the symplectic quotient of a unitary $\SU_2$-module; this computation will require the formulas for $\gamma_0$,
$\gamma_1$, and $\gamma_2$ presented here.

We briefly describe the notation required to state Theorem~\ref{thrm:Main}, which is adopted and explained in more detail
throughout the rest of the paper.
For $V = \bigoplus_{k=1}^r V_{d_k}$, let $D = \dim V$. Let
$\Lambda = \{(k,i)\in \Z\times\Z : 1 \leq k \leq r, \; \lfloor d_k/2 \rfloor + 1 \leq  i \leq d_k \}$ with $C:=|\Lambda|$,
and for each $(k,i)\in\Lambda$, let $a_{k,i} := 2i - d_k$. Let $\bs{a} := \big(a_{k,i} : (k,i)\in\Lambda\big)$,
and let $\sigma_V = 2$ if each $d_k$ is even and $1$ otherwise. Finally, for an integer partition
$\rho = (\rho_1,\ldots,\rho_n) \in \Z^n$ with $\rho_1\geq\rho_2\geq\cdots\rho_n\geq 0$, let
$s_\rho(\bs{x})$ denote the corresponding Schur polynomial in the variables $\bs{x}=(x_1,\ldots,x_n)$
(details can be found at the beginning of Section \ref{sec:LaurentSchur}).
We then have the following.

\begin{theorem}
\label{thrm:Main}
Let $V = \bigoplus_{k=1}^r V_{d_k}$ with $d_1 \le \cdots \le d_r$ be an $\SL_2$-representation with $V^{\SL_2} = \{0\}$,
and assume $V$ is not isomorphic to $2V_1$, nor $V_d$ for $d \leq 4$.
The degree $3-D$ coefficient $\gamma_0$ of the Laurent series of $\Hilb_V(t)$ is given by
\begin{equation}
\label{eq:Gamma0}
    \gamma_0
    = \sigma_V\,
    \frac{ s_{\rho}(\bs{a}) }
        {s_{\delta}(\bs{a})}
\end{equation}
where $\rho = (C-3,C-3,C-3,C-4,\ldots,1,0)$ and $\delta = (C-1,C-2,\ldots,1,0)$.
The degree $4-D$ coefficient $\gamma_1$ of the Laurent series is given by $3\gamma_0/2$.
Hence the $a$-invariant $-(D-3) - 2\gamma_1/\gamma_0$ of $\C[V]^{\SL_2}$ is equal to $-D$.

If $V$ is not isomorphic to $V_d$ for $d=1,2,3,4,5,6,8$, $2V_1$, $V_1+V_2$, $V_1+V_3$, $V_1+V_4$,
$2V_2$, $V_2+V_3$, $V_2+V_4$, $2V_3$, nor $2V_4$, then the degree $5-D$ coefficient $\gamma_2$
of the Laurent series is given as follows. If all $d_k$ are even, at least two $d_k$ are odd, or at least one odd $d_k > 1$, then
\begin{equation}
\label{eq:Gamma2Case1}
    \gamma_2
    =
    \sigma_V\,
    \frac{42 s_{\rho}(\bs{a}) + s_{\rho^\prime}(\bs{a})
            \big( P_2(\bs{a})  - 8\big) }
                {24 s_\delta(\bs{a})}
\end{equation}
where $P_2(\bs{a})$ denotes the power sum of degree $2$ and $\rho^\prime = (C-3, C-4, C-4, C-4, C-5,\ldots,1,0)$.
If $d_1 = 1$ and all other $d_k$ are even, then
\begin{align}
\label{eq:Gamma2Case2}
    \gamma_2
    &=
    \frac{42 s_{\rho}(\bs{a}) + s_{\rho^\prime}(\bs{a})
            \big( P_2(\bs{a})  - 8\big) }
                {24 s_\delta(\bs{a})}
        + \frac{s_{C - 4, C - 4, C - 4, C - 5, \ldots, 1, 0}(\bs{a_1}) }
            {4 s_{C - 2, C - 3, C - 4 \ldots, 1, 0}(\bs{a_1}) }
\end{align}
where $\bs{a}_1$ denotes $\bs{a}$ with the entry $a_{1,1}$ removed.

If $V$ is not isomorphic to $2V_1$, nor $V_d$ for $d \leq 4$, then
the degree $6-D$ coefficient $\gamma_3$ of the Laurent series is given by
\begin{equation}
\label{eq:Gamma3}
    \gamma_3 = \frac{5(\gamma_2 - \gamma_0)}{2}.
\end{equation}
\end{theorem}

Note that the second term in Equation~\eqref{eq:Gamma2Case2} can be interpreted as $\gamma_0^\prime/8$,
where $\gamma_0^\prime$ is the first Laurent coefficient of the Hilbert series associated to
$\bigoplus_{k=2}^r V_{d_k}$, unless $\bigoplus_{k=2}^r V_{d_k}$ is one of the exceptions for
Equation~\eqref{eq:Gamma0}.

After briefly discussing the relevant background in Section~\ref{sec:Back}, we turn to the computation
of the Hilbert series in Section~\ref{sec:HilbSer}. We first compute  an expression for the multivariate
Hilbert series, which in this case has no degeneracies, in Section~\ref{subsec:HilbSerMultivar},
and then demonstrate in Section~\ref{subsec:HilbSerAnalyticContin} the analytic continuation used to
state the univariate Hilbert series. We then turn to the computation of the Laurent coefficients
$\gamma_0$, $\gamma_1$, $\gamma_2$, and $\gamma_3$. The naive formulas for these, which only apply in the cases without
degeneracies, are computed in Section~\ref{sec:Laurent}; the removal of the singularities using
Schur polynomials is explained in Section~\ref{sec:LaurentSchur}.
The proof of Theorem~\ref{thrm:Main} is given in Section~\ref{sec:LaurentSchur} as
Theorems~\ref{thrm:Gamma0Schur}, \ref{thrm:Gamma1Schur}, and \ref{thrm:Gamma2Schur}
as well as Corollaries~\ref{cor:aInvar} and \ref{cor:Gamma3Schur}.
In Section~\ref{sec:Algorithm}, we describe an algorithm to compute the Hilbert series
based on Proposition~\ref{thrm:UnivarHilbSer}. Appendix~\ref{ap:Exceptions} lists
the Hilbert series and Laurent coefficients for each representation $V$ that is
an exception to some portion of Theorem~\ref{thrm:Main}.


\section*{Acknowledgements}
We would like to thank Gerald Schwarz for bringing to our attention the work of Friedrich Knop on the $a$-invariant
of invariant rings. Furthermore, we would like to thank Leonid Bedratyuk for pointing out references related to
this project. Herbig, Herden, and Seaton express appreciation to the Institute for Computational and Experimental
Research in Mathematics (ICERM), Herbig and Seaton express appreciation to Baylor University, and Herden
and Seaton express appreciation to the Instituto de Matem\'{a}tica Pura e Aplicada (IMPA) for
hospitality during the work contained in this manuscript. Herbig thanks CNPq for financial support.


\section{Background and Definitions}
\label{sec:Back}

Let $V_d$ denote the irreducible representation of $\SL_2$ of dimension $d + 1$ on binary forms
of degree $d$. Let $V$ be an arbitrary $\SL_2$-representation such that $V^{\SL_2} = \{0\}$.
Then $V$ is of the form
\[
    V = \bigoplus\limits_{k=1}^r V_{d_k}
\]
where each $d_k \geq 1$. Note that the $d_k$ need not be distinct.
We will assume for convenience that they are ordered non-decreasingly, i.e. $d_k \leq d_{k+1}$,
until Section~\ref{sec:Algorithm}, where it will be convenient to use a different notation.
Let $D$ denote the dimension of $V$, which is given by
\[
    D := \dim V = r + \sum\limits_{k=1}^r d_k.
\]
Let $\C[V]^{\SL_2}$ denote the algebra of $\SL_2$-invariant polynomial functions on $V$ with its usual $\N$-grading
by degree, and let $\Hilb_V(t) = \Hilb_{(d_1,\ldots, d_r)}(t)$ denote the univariate Hilbert series of
$\C[V]^{\SL_2}$. In Section~\ref{subsec:HilbSerMultivar}, we will also consider $\C[V]^{\SL_2}$ with the $\N^r$-grading
inherited from the decomposition $V = \bigoplus_{k=1}^r V_{d_k}$. That is, a monomial of degree
$(p_1,\ldots,p_r)$ is the product of monomials on $V_{d_k}$, each of degree $p_k$. We use
$\Hilb_V^r(t_1,\ldots,t_r) = \Hilb_{(d_1,\ldots, d_r)}^r(t_1,\ldots,t_r)$ to denote the corresponding
$r$-variate Hilbert series.

Using the Molien-Weyl formula \cite[Section 4.6.1]{DerskenKemperBook} and Weyl's Integration formula
\cite[Equation (26.19)]{FultonHarris}, the Hilbert series of $\C[V]^{\operatorname{SL}_2}$ can be expressed
as an integral over the Cartan torus of $\SL_2$. It will be helpful to define the constants $a_{k,i} := 2i - d_k$
for $k=1,\ldots,r$ and $0\leq i\leq d_k$, and then the Hilbert series $\Hilb_V(t)$ is given by the integral
\begin{equation}
\label{eq:MainIntFirst}
    \frac{1}{2\pi\sqrt{-1} }\int\limits_{\Sp^1}
        \frac{ (1 - z^2) \, dz}
        {z \prod\limits_{k=1}^r \prod\limits_{i=0}^{d_k} (1 - t z^{d_k - 2i}) }
        =
    \frac{1}{2\pi\sqrt{-1} }\int\limits_{\Sp^1}
        \frac{ (1 - z^2) \, dz}
        {z \prod\limits_{k=1}^r \prod\limits_{i=0}^{d_k} (1 - t z^{-a_{k,i}}) }.
\end{equation}
We will often use $c_k$ to denote a real parameter that is near $d_k$, in the sense that we will consider
the limit as the $c_k \to d_k$. Similarly, we will use $b_{k,i}$ to denote a real parameter near $a_{k,i}$.

For the case of the multivariate Hilbert series $\Hilb_V^r(t_1,\ldots,t_r)$, a simple modification to the proof
of the Molien-Weyl formula yields
\begin{equation}
\label{eq:MainIntMultivar}
    \Hilb_V^r(t_1,\ldots,t_r)
    =
    \frac{1}{2\pi\sqrt{-1} }\int\limits_{\Sp^1}
        \frac{ (1 - z^2) \, dz}
        {z \prod\limits_{k=1}^r \prod\limits_{i=0}^{d_k} (1 - t_k z^{-a_{k,i}}) }.
\end{equation}
See \cite[Equation (13)]{StanleyInvarFinGrp}, where this extension is given for the case of finite groups,
as well as \cite[Section IV]{Forger}, where it is described for the special bigraded case of a real
representation, where the bigrading considers the holomorphic and anti-holomorphic parts separately.

It will be convenient for us to use a few different methods to index the factors in the denominator
of the integral in Equation~\eqref{eq:MainIntFirst}. First, let us define
\[
    \Theta: = \{(k,i)\in \Z\times\Z : 1 \leq k \leq r, \; 0 \leq  i \leq d_k \},
\]
and then the integral in Equation~\eqref{eq:MainIntFirst} can be expressed as
\begin{equation}
\label{eq:MainIntTheta}
    \frac{1}{2\pi\sqrt{-1} }\int\limits_{\Sp^1}
        \frac{ (1 - z^2) \, dz}
        {z \prod\limits_{(k,i)\in\Theta} (1 - t z^{-a_{k,i}}) }.
\end{equation}
Note that $\Theta$ has $D$ elements.

We will sometimes wish to take advantage of the grouping of the nonzero $a_{k,i}$ into positive and
negative pairs. Hence, define $\Lambda$ to be the subset of $\Theta$ consisting of pairs $(k,i)$ such that
$a_{k,i} > 0$, i.e.
\[
    \Lambda = \{(k,i)\in \Z\times\Z : 1 \leq k \leq r, \; \lfloor d_k/2 \rfloor + 1 \leq  i \leq d_k \}.
\]
Let $C$ denote the cardinality of $\Lambda$,
\[
    C := \sum\limits_{k=1}^r \lceil d_k/2 \rceil,
\]
and let $e$ denote the number of $a_{k,i} = 0$, corresponding to the number of $k$ such that $d_k$ is even.
Then we can express Equation~\eqref{eq:MainIntTheta} as
\begin{equation}
\label{eq:MainIntLambda}
    \frac{1}{2\pi\sqrt{-1}}\int\limits_{\Sp^1}
        \frac{ (1 - z^2) \, dz}
        {z (1 - t)^e \prod\limits_{(k,i)\in\Lambda}
            (1 - t z^{-a_{k,i}})(1 - t z^{a_{k,i}})}.
\end{equation}


\section{Computation of the Hilbert Series}
\label{sec:HilbSer}


\subsection{The Multivariate Hilbert Series}
\label{subsec:HilbSerMultivar}

We first consider the computation of the multivariate Hilbert series and prove the following.
As noted above, this result is similar to that of Brion
\cite[Th\'{e}or\`{e}me 1]{Brion};
a computation of the multivariate Hilbert series in a different spirit was given by Bedratyuk and
Bedratyuk in \cite[Theorem 3]{BedratyukBedratyukMultivarPoincare}.

\begin{proposition}
\label{prop:MultivarHilbSer}
Let $V = \bigoplus_{k=1}^r V_{d_k}$ be an $\SL_2$-representation with $V^{\SL_2} = \{0\}$.
The $\N^r$-graded Hilbert series $\Hilb_V^r(t_1,\ldots,t_r)$ is given by
\begin{equation}
\label{eq:MultivarHilbSer}
    \sum\limits_{(K,I)\in\Lambda} \sum\limits_{\zeta^{a_{K,I}} = 1}
    \frac{ 1 - \zeta^2 t_K^{2/a_{K,I}}}
    {a_{K,I}(1 - t_K^2)
    \prod\limits_{d_k\in 2\Z}(1 - t_k)
        \prod\limits_{\substack{(k,i)\in\Lambda \smallsetminus \\ \{(K,I)\}}}
            \beta_{K,I,k,i,\zeta}^r(t_k, t_K)},
\end{equation}
where
$\beta_{K,I,k,i,\zeta}^r(t_k, t_K) := (1 - \zeta^{-a_{k,i}} t_k t_K^{-a_{k,i}/a_{K,I}})
(1 - \zeta^{a_{k,i}} t_k t_K^{a_{k,i}/a_{K,I}})$.
\end{proposition}
\begin{proof}
Using Equation~\eqref{eq:MainIntMultivar}, with the factors in the denominator indexed as in
Equation~\eqref{eq:MainIntLambda}, we have that $\Hilb_{(d_1,\ldots, d_r)}^r(t_1,\ldots,t_r)$
is equal to
\begin{align*}
    &\frac{1}{2\pi\sqrt{-1}}\int\limits_{\Sp^1}
        \frac{ (1 - z^2) \, dz}
        {z  \prod\limits_{d_k\in 2\Z}(1 - t_k) \prod\limits_{(k,i)\in\Lambda}
            (1 - t_k z^{-a_{k,i}})(1 - t_k z^{a_{k,i}})}\\
    & \quad =
    \frac{1}{2\pi\sqrt{-1}}\int\limits_{\Sp^1}
        \frac{(1 - z^2) z^{-1+\sum_{(k,i)\in\Lambda} a_{k,i}} \, dz}
        {  \prod\limits_{d_k\in 2\Z}(1 - t_k) \prod\limits_{(k,i)\in\Lambda}
            (z^{a_{k,i}} - t_k)(1 - t_k z^{a_{k,i}})}.
\end{align*}
Assume that each $|t_k| < 1$ for each $k$. Then the poles in $z$ inside the unit disk occur at points such that
$z^{a_{k,i}} = t_k$, i.e. points of the form $z = \zeta t_k^{1/a_{k,i}}$ where
$\zeta$ is an $a_{k,i}$th root of unity and the $ t_k^{1/a_{k,i}}$ are defined using
a suitably chosen, fixed branch of the logarithm. We assume that these poles are distinct,
which is true for a generic choice of the $t_k$.
Fix a $(K,I)\in\Lambda$ and an
$a_{K,I}$th root of unity $\zeta_0$, and then we express
\begin{align*}
    &\frac{ 1 - z^2}
    {z \prod\limits_{d_k\in 2\Z}(1 - t_k) \prod\limits_{(k,i)\in\Lambda}
        (1 - t_k z^{-a_{k,i}})(1 - t_k z^{a_{k,i}})}
    \\&\quad=
    \frac{ z^{a_{K,I} - 1} (1 - z^2)}
    {(z^{a_{K,I}} - t_K) (1 - t_K z^{a_{K,I}}) \prod\limits_{d_k\in 2\Z}(1 - t_k)
        \prod\limits_{\substack{(k,i)\in\Lambda\smallsetminus \\ \{(K,I)\}}}
            (1 - t_k z^{-a_{k,i}})(1 - t_k z^{a_{k,i}}) },
\end{align*}
with
\[
    z^{a_{K,I}} - t_K = (z - \zeta_0 t_K^{1/a_{K,I}})
        \prod\limits_{\substack{\zeta^{a_{K,I}} = 1\\ \zeta\neq\zeta_0}} (z - \zeta t_K^{1/a_{K,I}}).
\]
Hence, we have a simple pole at $z = \tau:= \zeta_0 t_K^{1/a_{K,I}}$, and the residue at $z = \tau$
is given by
\begin{align*}
    &\frac{ \tau^{a_{K,I} - 1} (1 - \tau^2)}
    {(1 - t_K \tau^{a_{K,I}})
        \prod\limits_{\substack{\zeta^{a_{K,I}} = 1\\ \zeta\neq\zeta_0}} (\tau - \zeta t_K^{1/a_{K,I}})
        \prod\limits_{d_k\in 2\Z}(1 - t_k) \prod\limits_{\substack{(k,i)\in\Lambda\smallsetminus \\ \{(K,I)\}}}
            (1 - t_k \tau^{-a_{k,i}})(1 - t_k \tau^{a_{k,i}}) }
    \\&=
    \frac{ \tau^{a_{K,I} - 1} (1 - \tau^2)}
    {(1 - t_K^2) \tau^{a_{K,I} - 1}
        \prod\limits_{\substack{\zeta^{a_{K,I}} = 1\\ \zeta\neq 1}} (1 - \zeta)
        \prod\limits_{d_k\in 2\Z}(1 - t_k) \prod\limits_{\substack{(k,i)\in\Lambda\smallsetminus \\ \{(K,I)\}}}
            (1 - t_k \tau^{-a_{k,i}})(1 - t_k \tau^{a_{k,i}}) }
    \\&=
    \frac{ 1 - \zeta_0^2 t_K^{2/a_{K,I}}}
    {a_{K,I}(1 - t_K^2) \prod\limits_{d_k\in 2\Z}(1 - t_k)
        \prod\limits_{\substack{(k,i)\in\Lambda\smallsetminus \\ \{(K,I)\}}}
            \beta_{K,I,k,i,\zeta_0}^r(t_k, t_K)
            }.
\end{align*}
Summing over each choice of $(K,I)$ and $\zeta_0$ completes the proof.
\end{proof}


\subsection{Analytic Continuation and the Univariate Hilbert Series}
\label{subsec:HilbSerAnalyticContin}

By the definition of the multivariate Hilbert series, it is clear that $\Hilb_V^r(\Delta_t) = \Hilb_V(t)$,
where $\Delta_t := (t,\ldots,t)\in \C^r$. However, the expression for $\Hilb_V^r(t_1,\ldots,t_r)$
given by Proposition~\ref{prop:MultivarHilbSer} is not defined after this substitution unless
$r = 1$ or $2$ and, when $r = 2$, one element of $\{d_1,d_2\}$ is even and the other is odd. One checks
that in all other cases, factors of the form $(1 - t_k t_K^{-1})$ appear in the denominator, e.g.
when $d_k$ and $d_K$ have the same parity so that for some choice of $i$ and $I$,
$a_{k,i} = a_{K,I}$ and $\zeta^{-a_{k,i}} = \zeta^{-a_{K,I}} = 1$. While it is again clear from the definitions that
$\lim_{(t_1,\ldots,t_r)\to\Delta_t} \Hilb_V^r(t_1,\ldots,t_r) = \Hilb_V(t)$,
we demonstrate explicitly in this section that the corresponding singularities in
Equation~\eqref{eq:MultivarHilbSer}
are removable, yielding an expression for the univariate Hilbert
series that is sufficiently explicit to compute the Laurent coefficients.
Note that Equation~\eqref{eq:MultivarHilbSer} has singularities at $t_k = t_K=t$ in the open unit disk
only where $\zeta^{-a_{k,i}} t^{(a_{K,I}-a_{k,i})/a_{K,I}} = 1$, which only occur in factors where $a_{K,I} = a_{k,i}$.
The argument in this section is similar to that of \cite[Section 3.3]{HerbigSeaton} and
\cite[Theorem 3.3]{CowieHerbigSeatonHerden}, the point here being that the same techniques extend
to the case of $G = \SL_2$ with very little modification.

To simplify the argument, we re-index as follows. Let $a$ be a positive value of $a_{i,j}$ that occurs with
multiplicity, set
\[
    \Lambda^a:= \{ (k,i)\in \Lambda : a_{k,i} \neq a \},
\]
and let $N$ be the cardinality of $\Lambda\smallsetminus\Lambda^a$. We consider the integral
\[
    \frac{1}{2\pi\sqrt{-1}}\int\limits_{\Sp^1}
        \frac{ (1 - z^2) \, dz}
        {z (1 - t)^e\prod\limits_{j=1}^N
                (1 - x_j z^{-a})(1 - x_j z^{a})
            \prod\limits_{(k,i)\in\Lambda^a}
                (1 - x_{k,i} z^{-a_{k,i}})(1 - x_{k,i} z^{a_{k,i}})}
\]
where the $x_j$ and $x_{k,i}$ are assumed distinct of modulus less than $1$ and contained in a fixed branch of
the logarithm. Because the integrand is defined and continuous and hence bounded for the $x_j$ and $x_{k,i}$
sufficiently close to $t$, an application of the Dominated Convergence Theorem demonstrates that the limit of
this integral as the $x_j \to t$ and $x_{k,i}\to t$, provided it exists, is equal to $\Hilb_V(t)$.

At a pole of the form $z = \zeta x_J^{1/a}$ where $\zeta$ is an $a$th root of unity, a computation identical to
that in Proposition~\ref{prop:MultivarHilbSer} yields that the residue is given by
\[
    \frac{ 1 - \zeta^2 x_J^{2/a}}
    {a (1 - t)^e (1 - x_J^2)
    \prod\limits_{\substack{j=1\\ j\neq J}}^N
    (1 - x_j x_J^{-1}) (1 - x_j x_J)
        \prod\limits_{(k,i)\in\Lambda^a}
            \beta_{K,I,k,i,\zeta}^r(x_{k,i}, x_J)}.
\]
Rewrite
\[
    \frac{ x_J^{N-1} (1 - \zeta^2 x_J^{2/a})}
    {a(1 - t)^e (1 - x_J^2)
    \prod\limits_{\substack{j=1\\ j\neq J}}^N
    (x_J - x_j) (1 - x_j x_J)
        \prod\limits_{(k,i)\in\Lambda^a}
            \beta_{K,I,k,i,\zeta}^r(x_{k,i}, x_J)},
\]
and consider the sum of the residues at $z = \zeta x_J^{1/a}$ where $J$ ranges from $1$ to $N$ and
$\zeta$ remains fixed, i.e.
\begin{align*}
    &\sum\limits_{J=1}^N \frac{ x_J^{N-1} (1 - \zeta^2 x_J^{2/a})}
    {a(1 - t)^e (1 - x_J^2)
    \prod\limits_{\substack{j=1\\ j\neq J}}^N
    (x_J - x_j) (1 - x_j x_J)
        \prod\limits_{(k,i)\in\Lambda^a}
            \beta_{K,I,k,i,\zeta}^r(x_{k,i}, x_J)}
    \\ &=
    \frac{1}
    {a(1 - t)^e\prod\limits_{p=1}^N(1 - x_p^2)
    \prod\limits_{1\leq p < q \leq N}
    (x_p - x_q) (1 - x_p x_q)
        \prod\limits_{\substack{(k,i)\in\Lambda^a\\ 1\le j \le N}}
            \beta_{K,I,k,i,\zeta}^r(x_{k,i}, x_j)}\cdot\\
    & \quad\quad \sum\limits_{J=1}^N  \left((-1)^{J-1} x_J^{N-1} (1 - \zeta^2 x_J^{2/a})
     \prod\limits_{\substack{p=1\\ p\neq J}}^N(1 - x_p^2)
        \prod\limits_{\substack{{1\leq p < q \leq N} \\ p,q\neq J}}
            (x_p - x_q) (1 - x_p x_q)\cdot \right. \\ & \quad\quad\quad
        \left.
        \prod\limits_{\substack{(k,i)\in\Lambda^a\\ 1\le j \le N, j\ne J}}
            \beta_{K,I,k,i,\zeta}^r(x_{k,i}, x_j)\right).
\end{align*}
The numerator is easily seen to be alternating in the $x_J$, implying that it is divisible
by the Vandermonde determinant $\prod_{1\leq p < q \leq N}(x_p - x_q)$ in the denominator.
That is, for each fixed $\zeta$, the sum of residues at poles of the form $z = \zeta x_J^{1/a}$
has removable singularities in the $x_J$ at points where each $x_J = t$. Applying this argument
to each value of $a_{k,i}$ that occurs in $\Lambda$ with multiplicity, it follows that the
limit of Equation~\eqref{eq:MultivarHilbSer} as each $t_j\to t$ exists. Finally, by a series of
simple substitutions identical to those used in \cite[pages 52--53]{HerbigSeaton} and
\cite[proof of Theorem 3.3]{CowieHerbigSeatonHerden}, we have the following.

\begin{theorem}
\label{thrm:UnivarHilbSer}
Let $V = \bigoplus_{k=1}^r V_{d_k}$ be an $\SL_2$-representation with $V^{\SL_2} = \{0\}$.
Let $\bs{a} := \big(a_{k,i} : (k,i)\in\Lambda\big)$, and let
$\bs{b} := \big(b_{k,i} : (k,i)\in\Lambda\big)$ where the $b_{k,i}$ are
real parameters. The $\N$-graded Hilbert series is given by
\begin{equation}
\label{eq:UnivarHilbSer}
    \Hilb_V(t) =
    \lim\limits_{\bs{b}\to\bs{a}}
    \sum\limits_{(K,I)\in\Lambda} \sum\limits_{\zeta^{a_{K,I}} = 1}
    \frac{ 1 - \zeta^2 t^{2/b_{K,I}}}
    {b_{K,I}(1 - t^2)(1 - t)^e
        \prod\limits_{\substack{(k,i)\in\Lambda\smallsetminus \\ \{(K,I)\}}}
            \beta_{K,I,k,i,\zeta}(t)
            },
\end{equation}
where $\beta_{K,I,k,i,\zeta}(t) := (1 - \zeta^{-a_{k,i}} t^{(b_{K,I}-b_{k,i})/b_{K,I}})
(1 - \zeta^{a_{k,i}} t^{(b_{K,I}+b_{k,i})/b_{K,I}})$.

Alternatively, using $\bs{a}_\Theta := \big(a_{k,i} : (k,i)\in\Theta\big)$ and real
parameters $\bs{b}_\Theta := \big(b_{k,i} : (k,i)\in\Theta\big)$,
the $\N$-graded Hilbert series is given by
\begin{equation}
\label{eq:UnivarHilbSerTheta}
    \Hilb_V(t) =
    \lim\limits_{\bs{b}_\Theta\to\bs{a}_\Theta}
    \sum\limits_{(K,I)\in\Lambda} \sum\limits_{\zeta^{a_{K,I}} = 1}
    \frac{ 1 - \zeta^2 t^{2/b_{K,I}}}
    {b_{K,I}
    \prod\limits_{\substack{(k,i)\in\Theta\smallsetminus \\ \{(K,I)\}}}
            (1 - \zeta^{-a_{k,i}} t^{(b_{K,I}-b_{k,i})/b_{K,I}})
            }.
\end{equation}
\end{theorem}


\section{The Coefficients of the Laurent Expansion}
\label{sec:Laurent}

In this section, we compute the first four coefficients of the Laurent expansion at $t = 1$ of the
expression given in Equation~\eqref{eq:UnivarHilbSerTheta}. We will see in Section~\ref{sec:LaurentSchur} that,
after taking the limit $\bs{b}_\Theta\to\bs{a}_\Theta$, these coefficients correspond to the Laurent
coefficients $\gamma_0$, $\gamma_1$, $\gamma_2$, and $\gamma_3$ of $\Hilb_V(t)$, see Equation~\eqref{eq:DefGammas}.
Throughout this section, it will be convenient to index the factors of the integral as in
Equation~\eqref{eq:MainIntTheta}. Hence, for $(K,I)\in\Lambda$ and $\zeta$ an $a_{K,I}$th root
of unity, we define
\begin{equation}
\label{eq:DefH}
    H_{V,K,I,\zeta}(\bs{b}_\Theta,t):=   \frac{ 1 - \zeta^2 t^{2/b_{K,I}}}
    {b_{K,I} \prod\limits_{\substack{(k,i)\in\Theta\smallsetminus \\ \{(K,I)\}}}
            (1 - \zeta^{-a_{k,i}} t^{(b_{K,I}-b_{k,i})/b_{K,I}})
            }
\end{equation}
so that by Theorem~\ref{thrm:UnivarHilbSer},
\begin{equation}
\label{eq:DefH2}
    \Hilb_V(t)  =   \lim\limits_{\bs{b}_\Theta\to\bs{a}_\Theta}
        \sum\limits_{(K,I)\in\Lambda} \sum\limits_{\zeta^{a_{K,I}} = 1}
        H_{V,K,I,\zeta}(\bs{b}_\Theta,t).
\end{equation}
Our method will be to consider the Laurent expansions of each of the terms
$H_{V,K,I,\zeta}(\bs{b}_\Theta,t)$ separately.


\subsection{Discussion of Cases}
\label{subsec:LaurentCases}

Assume $V$ is $1$-large, which is true unless $V$ is isomorphic to $V_1$, $2V_1$, or $V_2$ by
\cite[Theorem 3.4]{HerbigSchwarz}; we refer the reader to this reference for the definition
of $1$-large. Then $\C[V]^{\SL_2}$ has Krull dimension $D - 3$,
see \cite[Remark 9.2(3)]{GWSlifting}. Hence, the first nontrivial Laurent coefficient occurs in
degree $3-D$.

Now, any term of the form $H_{V,K,I,1}(\bs{b}_\Theta, t)$ for $(K,I)\in\Lambda$ is of the form
\begin{equation}
\label{eq:TermZ1}
    H_{V,K,I,1}(\bs{b}_\Theta, t) =
    \frac{ 1 - t^{2/b_{K,I}}}
    {b_{K,I} \prod\limits_{\substack{(k,i)\in\Theta\smallsetminus \\ \{(K,I)\}}}
            (1 - t^{(b_{K,I} - b_{k,i})/b_{K,I}})}.
\end{equation}
Each such term has a pole of order $|\Theta|-2=D - 2$ at $t = 1$. If each $d_k$ is even, then each
$a_{k,i}$ is even, and then for each $(K,I)\in\Lambda$, we have
$H_{V,K,I,1}(\bs{b}_\Theta, t) = H_{V,K,I,-1}(\bs{b}_\Theta, t)$. To simplify notation, we define
\begin{equation}
\label{eq:DefPV}
    \sigma_V :=
    \begin{cases}
        1,  &   \mbox{if any $d_k$ is odd,}         \\
        2,  &   \mbox{if all $d_k$ are even.}
    \end{cases}
\end{equation}

If $d_1 = 1$ and all other $d_k$ are even, then for $K > 1$ and any $I$
the term $H_{V,K,I,-1}(\bs{b}_\Theta, t)$ is given by
\begin{align}
\label{eq:Case2TermNot1}
    &
    (1 - t^{2/b_{K,I}})
        \Big(b_{K,I} (1 + t^{(b_{K,I}-b_{1,0})/b_{K,I}})
            (1 + t^{(b_{K,I}-b_{1,1})/b_{K,I}}) \cdot
    \\ \nonumber &\quad\quad
            \prod\limits_{\substack{(k,i)\in\Theta\smallsetminus \\ \{(K,I),(1,0),(1,1)\} }}
                (1 - t^{(b_{K,I}-b_{k,i})/b_{K,I}}) \Big)^{-1}.
\end{align}
which has a pole of order $D - 4$. In particular, in Equation \eqref{eq:DefH2}
any term of the form $H_{V,K,I,-1}$ would have a pole of order $D - 4$.
In any other case, i.e. if two or more $d_k$ are odd or one $d_k > 1$ is odd,
then any $H_{V,K,I,-1}$ appearing in Equation \eqref{eq:DefH2} has a pole
of order at most $D - 6$.

Now, if $\zeta\neq\pm 1$ is an $a_{K,I} = (2I - d_K)$th root of unity, then it must be that
$|a_{K,I}| \geq 3$, and hence $d_K \geq 3$. The numerator $1 - \zeta^2 t^{2/b_{K,I}}$ of
$H_{V,K,I,\zeta}(\bs{b}_\Theta,t)$ no longer has a zero at $t = 1$. For each odd $d_k$, we have
$a_{k,(d_k\pm 1)/2} = \pm 1$,
while for each even $d_k$, we have $a_{k,(d_k/2) \pm 1} = \pm 2$, and
$\zeta^{\pm 1}, \zeta^{\pm 2} \neq 1$. Hence, if $r \geq 2$, then in Equation \eqref{eq:DefH2}
each such term has a pole of
order at most $D - 5$. If $r = 1$ and $d_1 = 3$ or $4$, i.e. $V \cong V_3$ or $V \cong V_4$,
then some terms $H_{V,K,I,\zeta}(\bs{b}_\Theta, t)$ in Equation \eqref{eq:DefH2} corresponding to
such a $\zeta$ may have pole order $D-3$. If $r = 1$ and $d_1 \ge 5$,
it is easy to see that at least four of the $a_{k,i}$ satisfy $\zeta^{a_{k,i}} \neq 1$ so that the
pole order of such a term is at most $D - 5$.

Hence, in the computation of $\gamma_0$, the coefficient of the Laurent series at $t = 1$ of degree $3-D$,
and $\gamma_1$, the coefficient of degree $4-D$, we must consider  $V_1$, $2V_1$, $V_2$, $V_3$, and $V_4$ as exceptions. In all other cases, only terms corresponding to $\zeta = \pm 1$ contribute to $\gamma_0$ and $\gamma_1$.
For $\gamma_1$, we must also consider $V_1 + V_2$ as an exception, because $H_{V,K,I,\pm 1}(\bs{b}_\Theta, t)$
in this case has too few factors for our arguments to apply.

In the computation of the coefficient $\gamma_2$ of degree $5-D$, there are more exceptions.
By considering cases as above, it is easy to see that for $\zeta\neq\pm 1$, the expression
$H_{V,K,I,\zeta}(\bs{b}_\Theta, t)$ in Equation \eqref{eq:DefH2} may have a pole of order
$D-5$ when $r = 1$ and $d = 1, 2, 3, 4, 5, 6, 8$, or when $r = 2$ in the cases
$2V_1$, $V_1+V_3$, $V_1+V_4$, $V_2+V_3$, $V_2+V_4$, $2V_3$, and $2V_4$. For any other case,
$H_{V,K,I,\zeta}(\bs{b}_\Theta, t)$ has a pole order of at most $D-7$ unless $\zeta=\pm 1$.
We also exclude $V_1+V_2$ and $2V_2$ as exceptions, because $H_{V,K,I,\pm 1}(\bs{b}_\Theta, t)$
has too few factors.

Excluding these exceptions, terms of the form $H_{V,K,I,-1}(\bs{b}_\Theta, t)$ contribute to
$\gamma_0$ only if each $d_k$ is even, in which case they are identical to the
corresponding $H_{V,K,I,1}(\bs{b}_\Theta, t)$ and hence contribute to $\gamma_0$, $\gamma_1$, and $\gamma_2$
in the same way.
If there are at least two odd $d_k$ or one odd $d_k > 1$, then
$H_{V,K,I,-1}(\bs{b}_\Theta, t)$ has a pole of order at most $D-6$ so that
only terms corresponding to $\zeta = 1$ contribute to $\gamma_0$, $\gamma_1$, and $\gamma_2$. If $d_1 = 1$ and all other $d_k$ are even,
then terms with $\zeta = -1$ contribute to $\gamma_1$ and $\gamma_2$.

It is of interest to note that the Laurent series of $H_{V,K,I,1}(\bs{b}_\Theta, t)$ has a term of degree $2-D$ with
coefficient
\[
    \frac{2b_{K,I}^{D-3} }
        {\prod\limits_{\substack{(k,i)\in\Theta\smallsetminus \\ \{(K,I)\} }}
            (b_{K,I}-b_{k,i}) }.
\]
Hence, by the above observations, we have the following.

\begin{corollary}
\label{cor:FirstCoeffVanish}
Let $V = \bigoplus_{k=1}^r V_{d_k}$ be an $\SL_2$-representation with $V^{\SL_2} = \{0\}$,
and assume that $V$ is not isomorphic to $V_1$, $2V_1$, nor $V_2$. Then
\begin{equation}
\label{eq:CorD2}
    \lim\limits_{\bs{b}_\Theta\to\bs{a}_\Theta} \sum\limits_{(K,I)\in\Lambda}
    \frac{2b_{K,I}^{D-3} }
        {\prod\limits_{\substack{(k,i)\in\Theta\smallsetminus \\ \{(K,I)\} }}
            (b_{K,I}-b_{k,i}) }
    = 0.
\end{equation}
\end{corollary}
Note that the quantity on the left side of Equation~\eqref{eq:CorD2} is equal to $1$ when
$V\cong V_1$, equal to $-1/4$ when $V\cong V_2$, and equal to $-1$ when $V\cong 2V_1$.


\subsection{$\gamma_0$}
\label{subsec:LaurentGamma0}

In this section, we compute the degree $3-D$ coefficient of the Laurent series of
$\sum_{(K,I)\in\Lambda} \sum_{\zeta^{a_{K,I}} = 1} H_{V,K,I,\zeta}(\bs{b}_\Theta,t)$.
In Section~\ref{sec:LaurentSchur}, we will use that $\gamma_0$ is the limit of
the expression computed here as $\bs{b}_\Theta\to\bs{a}_\Theta$.

\begin{proposition}
\label{prop:Gamma0First}
Let $V = \bigoplus_{k=1}^r V_{d_k}$ be an $\SL_2$-representation with $V^{\SL_2} = \{0\}$,
and assume $V$ is not isomorphic to $2V_1$, nor $V_d$ for $d \leq 4$.
Let $\bs{a}_\Theta := \big(a_{k,i} : (k,i)\in\Theta\big)$, and let
$\bs{b}_\Theta := \big(b_{k,i} : (k,i)\in\Theta\big)$ where the $b_{k,i}$ are
real parameters. The degree $3-D$ coefficient of the Laurent series of
$\sum_{(K,I)\in\Lambda} \sum_{\zeta^{a_{K,I}} = 1} H_{V,K,I,\zeta}(\bs{b}_\Theta,t)$
is given by
\begin{align}
\label{eq:Gamma0First}
    &\sigma_V \sum\limits_{(K,I) \in \Lambda} \left(
        \frac{b_{K,I}^{D-3} - 2 b_{K,I}^{D-4}
            - \sum\limits_{\substack{ (\kappa,\lambda)\in\Theta\smallsetminus \\ \{(K,I)\} }}
                b_{K,I}^{D-4} b_{\kappa,\lambda} }
        {\prod\limits_{\substack{(k,i)\in\Theta\smallsetminus \\ \{(K,I)\} }}
            (b_{K,I} - b_{k,i}) }
            \right)
    \\ \nonumber &\quad\quad =
    \sigma_V \sum\limits_{(K,I) \in \Lambda} b_{K,I}^{D-4}\
        \frac{2b_{K,I} - 2
        - \sum\limits_{(k,i)\in\Theta} b_{k,i}}
        {\prod\limits_{\substack{(k,i)\in\Theta\smallsetminus \\ \{(K,I)\} }}(b_{K,I} - b_{k,i})}.
\end{align}
\end{proposition}
\begin{proof}
As explained in Section~\ref{subsec:LaurentCases}, the only terms that contribute to the degree $3-D$
coefficient are of the form given in Equation~\eqref{eq:TermZ1}
corresponding to $H_{V,K,I,1}(\bs{b}_\Theta,t)$ in each case and $H_{V,K,I,-1}(\bs{b}_\Theta,t)$ only if each
$d_k$ is even, yielding the $\sigma_V$ factor. The series expansion of the numerator $1 - t^{2/b_{K,I}}$ begins
\begin{align}
\label{eq:SeriesHolo}
    1 - t^{2/b_{K,I}}
        &=  \frac{2}{b_{K,I}} (1 - t) + \frac{b_{K,I} - 2}{b_{K,I}^2} (1 - t)^2
            + \frac{2(b_{K,I}^2 - 3b_{K,I} + 2)}{3b_{K,I}^3} (1 - t)^3
        \\ \nonumber &
            + \frac{3b_{K,I}^3 - 11b_{K,I}^2 + 12b_{K,I} - 4}{6b_{K,I}^4} (1 - t)^4
            + \cdots,
\end{align}
and each factor of the denominator has a series expansion beginning
\begin{align}
\label{eq:SeriesPole}
    \frac{1}{1 - t^{(b_{K,I} - b_{k,i})/b_{K,I}}}
        &=  \frac{b_{K,I}}{b_{K,I} - b_{k,i}} (1 - t)^{-1} - \frac{b_{k,i}}{2(b_{K,I} - b_{k,i})}
        \\ \nonumber &\quad\quad
            - \frac{(2 b_{K,I} - b_{k,i})b_{k,i}}{12(b_{K,I} - b_{k,i})b_{K,I} } (1 - t)
        \\ \nonumber &\quad\quad
            - \frac{(2 b_{K,I} - b_{k,i})b_{k,i}}{24(b_{K,I} - b_{k,i})b_{K,I} } (1 - t)^2
            + \cdots.
\end{align}
Hence, for fixed $K$ and $I$, the term of degree $3 - D$ comes from the Cauchy product of these factors
in two different ways. First, from the degree $1$ term of the holomorphic series in Equation~\eqref{eq:SeriesHolo},
the degree $-1$ term of the Laurent series in Equation~\eqref{eq:SeriesPole} for each value of $(k, i)$ except one
(say $(\kappa, \lambda)$), and the degree $0$ term of the Laurent series in Equation~\eqref{eq:SeriesPole} corresponding
to $(k,i) = (\kappa,\lambda)$. Second, from the degree $2$ term of the series in Equation~\eqref{eq:SeriesHolo}
and the degree $-1$ term of the Laurent series in Equation~\eqref{eq:SeriesPole} for each value of $(k, i)$.

The first of these combinations yields
\begin{align*}
    &\left( \frac{1}{b_{K,I}} \right)
    \left( \frac{2}{b_{K,I}} \right)
    \left( \prod\limits_{\substack{(k,i)\in\Theta\smallsetminus \\ \{(K,I), (\kappa,\lambda)\} }}
            \frac{b_{K,I}}{b_{K,I} - b_{k,i}} \right)
    \left( \frac{b_{\kappa,\lambda}}{2(b_{\kappa,\lambda} - b_{K,I})} \right)
    \\&\quad\quad=
    \frac{-b_{K,I}^{D-4} b_{\kappa,\lambda}}
        {\prod\limits_{\substack{(k,i)\in\Theta\smallsetminus \\ \{(K,I)\}}} (b_{K,I} - b_{k,i})},
\end{align*}
while the second yields
\[
    \left( \frac{1}{b_{K,I}} \right)
    \left( \frac{b_{K,I} - 2}{b_{K,I}^2} \right)
    \left( \prod\limits_{\substack{(k,i)\in\Theta\smallsetminus \\ \{(K,I)\}}}
            \frac{b_{K,I}}{b_{K,I} - b_{k,i}} \right)
    =
    \frac{b_{K,I}^{D-3} - 2 b_{K,I}^{D-4}}
        {\prod\limits_{\substack{(k,i)\in\Theta\smallsetminus \\ \{(K,I)\}}}
            (b_{K,I} - b_{k,i}) },
\]
completing the proof.
\end{proof}

Note that if $V = V_d$ is irreducible with $d \geq 5$, setting each $b_{1,i} = 2i - d$,
the expression in Equation~\eqref{eq:Gamma0First} reduces to
\[
    2\sigma_V \sum\limits_{I = \lfloor d/2\rfloor + 1}^d
        \frac{(2I - d)^{d-3} (2I - d - 1)}
            {\prod\limits_{\substack{i = 0 \\ i\neq I}}^d (2I - 2i)}.
\]
A simple computation demonstrates that this is equal to the expression in
Equation~\eqref{eq:Hilbert} given by Hilbert in \cite{Hilbert}.


\subsection{$\gamma_1$}
\label{subsec:LaurentGamma1First}

We now compute the second coefficient of the Laurent expansion of the expression
$\sum_{(K,I)\in\Lambda} \sum_{\zeta^{a_{K,I}} = 1} H_{V,K,I,\zeta}(\bs{b}_\Theta,t)$.

\begin{proposition}
\label{prop:Gamma1First}
Let $V = \bigoplus_{k=1}^r V_{d_k}$ be an $\SL_2$-representation with $V^{\SL_2} = \{0\}$,
and assume $V$ is not isomorphic to $2V_1$, $V_1+V_2$ nor $V_k$ for $k \leq 4$.
Let $\bs{a}_\Theta := \big(a_{k,i} : (k,i)\in\Theta\big)$, and let
$\bs{b}_\Theta := \big(b_{k,i} : (k,i)\in\Theta\big)$ where the $b_{k,i}$ are
real parameters.

If all $d_k$ are even, at least two $d_k$ are odd, or at least one odd $d_k > 1$, then the degree $4-D$ coefficient
of the Laurent series of $\sum_{(K,I)\in\Lambda} \sum_{\zeta^{a_{K,I}} = 1} H_{V,K,I,\zeta}(\bs{b}_\Theta,t)$ is given by
\begin{align}
\label{eq:Gamma1FirstCase1}
    &\sigma_V\sum\limits_{(K,I)\in\Lambda}
        \frac{b_{K,I}^{D-5}}
        {\prod\limits_{\substack{(k,i)\in\Theta\smallsetminus \\ \{(K,I)\}}} (b_{K,I} - b_{k,i})}
        \left(
        \vphantom{\sum\limits_{\substack{(\kappa^\prime,\lambda^\prime)\in\Theta\smallsetminus
                            \\ \{(K,I),(\kappa,\lambda)\}}} b_{\kappa^\prime,\lambda^\prime}}
        \frac{2}{3} (b_{K,I}^2 - 3b_{K,I} + 2)
                \right.
                \\ \nonumber &\quad\quad\quad\quad\quad\quad + \left.
                \sum\limits_{\substack{(\kappa,\lambda)\in\Theta\smallsetminus \\ \{(K,I)\}}}
                    \frac{b_{\kappa,\lambda}}{6} \left(
                        b_{\kappa,\lambda} - 5b_{K,I} + 6 + 3
                    \sum\limits_{\substack{(\kappa^\prime,\lambda^\prime)\in\Theta\smallsetminus
                            \\ \{(K,I),(\kappa,\lambda)\}}} b_{\kappa^\prime,\lambda^\prime}
        \right)\right).
\end{align}

If $d_1 = 1$ and all other $d_k$ are even, then the degree $4-D$ coefficient of the
Laurent series is given by
\begin{align}
    \label{eq:Gamma1FirstCase2}
    &\sum\limits_{(K,I)\in\Lambda}
         {\prod\limits_{\substack{(k,i)\in\Theta\smallsetminus \\ \{(K,I)\}}} (b_{K,I} - b_{k,i})}
        \left(
        \vphantom{\sum\limits_{\substack{(\kappa^\prime,\lambda^\prime)\in\Theta\smallsetminus
                            \\ \{(K,I),(\kappa,\lambda)\}}} b_{\kappa^\prime,\lambda^\prime}}
        \frac{2}{3} (b_{K,I}^2 - 3b_{K,I} + 2)
                \right.
                \\ \nonumber
                &\quad\quad\quad\quad\quad\quad + \left.
                \sum\limits_{\substack{(\kappa,\lambda)\in\Theta\smallsetminus \\ \{(K,I)\}}}
                    \frac{b_{\kappa,\lambda}}{6} \left(
                        b_{\kappa,\lambda} - 5b_{K,I} + 6 + 3
                    \sum\limits_{\substack{(\kappa^\prime,\lambda^\prime)\in\Theta\smallsetminus
                            \\ \{(K,I),(\kappa,\lambda)\}}} b_{\kappa^\prime,\lambda^\prime}
        \right)\right)
    \\ \nonumber &\quad\quad\quad\quad\quad\quad\quad
        +   \sum\limits_{\substack{(K,I)\in\Lambda\smallsetminus \\ \{ (1,0), (1,1) \} }}
            \frac{ b_{K,I}^{D-5} }
                { 2 \prod\limits_{\substack{(k,i)\in\Theta\smallsetminus \\ \{ (K,I), (1,0), (1,1) \} }}
                    (b_{K,I}-b_{k,i})}.
\end{align}
\end{proposition}
\begin{proof}
We first consider the case where all $d_k$ are even, at least two $d_k$ are odd or at least one $d_k > 1$ is odd.
Once again, it was explained in Section~\ref{subsec:LaurentCases} that the only contributing terms are of the form
given in Equation~\eqref{eq:TermZ1}. To compute the degree $4 - D$ coefficient, we consider the factors of the term
in Equation~\eqref{eq:TermZ1} and use the expansions given by Equations~\eqref{eq:SeriesHolo} and \eqref{eq:SeriesPole}.
A term of degree $4 - D$ can arise from the Cauchy product formula from these factors in one of four ways:
\begin{enumerate}
\item   The degree $1$ term from the numerator, the degree $-1$ term from each factor of the
        denominator except two, say $(\kappa,\lambda), (\kappa^\prime,\lambda^\prime)\neq (K,I)$,
        and the degree $0$ term from the factors of the denominator corresponding to $(\kappa,\lambda)$ and
        $(\kappa^\prime,\lambda^\prime)$;

\item   The degree $1$ term from the numerator, the degree $-1$ term from each factor of the denominator except
        one, say $(\kappa,\lambda)\neq (K,I)$, and the degree $1$ term from the factor of the denominator corresponding
        to $(\kappa,\lambda)$;

\item   The degree $2$ term from the numerator, the degree $-1$ term from each factor of the denominator except
        one, say $(\kappa,\lambda)\neq (K,I)$, and the degree $0$ term from the factor of the denominator corresponding
        to $(\kappa,\lambda)$; and

\item   The degree $3$ term from the numerator and the degree $-1$ term from each factor of the denominator.
\end{enumerate}

Computing as in the proof of Proposition~\ref{prop:Gamma0First}, the case (1) yields
\[
    \frac{b_{K,I}^{D-5} b_{\kappa,\lambda}b_{\kappa^\prime,\lambda^\prime}}
        {2\prod\limits_{\substack{(k,i)\in\Theta\smallsetminus \\ \{ (K,I)\} }} (b_{K,I} - b_{k,i})}.
\]
In (2), we have
\[
    \frac{- b_{K,I}^{D-5} (2b_{K,I} - b_{\kappa,\lambda}) b_{\kappa,\lambda}}
        {6\prod\limits_{\substack{(k,i)\in\Theta\smallsetminus \\ \{ (K,I)\} }} (b_{K,I} - b_{k,i})}.
\]
For (3),
\[
    \frac{- b_{K,I}^{D-5} (b_{K,I} - 2)b_{\kappa,\lambda} }
        {2\prod\limits_{\substack{(k,i)\in\Theta\smallsetminus \\ \{ (K,I)\} }} (b_{K,I} - b_{k,i})},
\]
and (4) yields
\[
    \frac{2 b_{K,I}^{D-5} (b_{K,I}^2 - 3b_{K,I} + 2) }
        {3\prod\limits_{\substack{(k,i)\in\Theta\smallsetminus \\ \{ (K,I)\} }} (b_{K,I} - b_{k,i})}.
\]
Combining these and summing over $(K,I)\in\Lambda$ completes the proof of Equation~\eqref{eq:Gamma1FirstCase1}.

We now assume $d_1 = 1$ and each $d_k$ for $k > 1$ is even, and then $\sigma_V = 1$. The terms with $\zeta = 1$ are
identical to the previous case, while terms with $K > 1$ and $\zeta = -1$ are of the form giving in
Equation~\eqref{eq:Case2TermNot1}. The degree $4-D$ coefficient of the Laurent series of such a latter
term is given by
\[
    \left(
    \frac{1}{ 4 b_{K,I} }
    \right)
    \left(
    \frac{2}{ b_{K,I} }
    \right)
    \left(
    \prod\limits_{\substack{(k,i)\in\Theta\smallsetminus \\ \{ (K,I), (1,0), (1,1) \} }}
        \frac{ b_{K,I} }{ b_{K,I}-b_{k,i} }
    \right)
    =
    \frac{ b_{K,I}^{D-5} }
        { 2 \prod\limits_{\substack{(k,i)\in\Theta\smallsetminus \\ \{ (K,I), (1,0), (1,1) \} }}
            (b_{K,I}-b_{k,i})}.
\]
Summing over all $(K,I)$ with $K > 1$ yields Equation~\eqref{eq:Gamma1FirstCase2}.
\end{proof}


\subsection{$\gamma_2$}
\label{subsec:LaurentGamma2First}

In this section, we turn to the computation of the degree $5-D$ coefficient of the Laurent expansion.

\begin{proposition}
\label{prop:Gamma2First}
Let $V = \bigoplus_{k=1}^r V_{d_k}$ be an $\SL_2$-representation with $V^{\SL_2} = \{0\}$,
and assume $V$ is not isomorphic to $V_d$ for $d=1,2,3,4,5,6,8$, $2V_1$, $V_1+V_2$, $V_1+V_3$, $V_1+V_4$,
$2V_2$, $V_2+V_3$, $V_2+V_4$, $2V_3$, nor $2V_4$. Let $\bs{a}_\Theta := \big(a_{k,i} : (k,i)\in\Theta\big)$, and let
$\bs{b}_\Theta := \big(b_{k,i} : (k,i)\in\Theta\big)$ where the $b_{k,i}$ are
real parameters.

If all $d_k$ are even, at least two $d_k$ are odd, or at least one odd $d_k > 1$, then the degree $5-D$ coefficient
of the Laurent series of $\sum_{(K,I)\in\Lambda} \sum_{\zeta^{a_{K,I}} = 1} H_{V,K,I,\zeta}(\bs{b}_\Theta,t)$ is given by
\begin{align}
    \label{eq:Gamma2FirstCase1}
    &\sigma_V\sum\limits_{(K,I)\in\Lambda}
        \frac{b_{K,I}^{D-6}}
        {24 \prod\limits_{\substack{(k,i)\in\Theta\smallsetminus \\ \{ (K,I)\} }} (b_{K,I} - b_{k,i})}
        \bigg(  12b_{K,I}^3 - 44b_{K,I}^2 + 48b_{K,I} - 16
                \\ \nonumber &\quad\quad\quad +
                \sum\limits_{\substack{(\kappa,\lambda)\in\Theta\smallsetminus \\ \{ (K,I)\} }} b_{\kappa,\lambda} \Big(
                    - 16b_{K,I}^2 + 32 b_{K,I} - 16
                        - 4 b_{\kappa,\lambda} + 4 b_{K,I} b_{\kappa,\lambda}
                \\ \nonumber &\quad\quad\quad
                    + \sum\limits_{\substack{(\kappa^\prime,\lambda^\prime)\in\Theta\smallsetminus
                            \\ \{(K,I),(\kappa,\lambda)\}}} b_{\kappa^\prime,\lambda^\prime} \Big(
                        7 b_{K,I}
                        - 6
                        - 2 b_{\kappa,\lambda}
                        -  \sum\limits_{\substack{(\kappa^{\prime\prime},\lambda^{\prime\prime})\in\Theta\smallsetminus
                            \\ \{(K,I),(\kappa,\lambda),(\kappa^{\prime},\lambda^{\prime})\}}}
                            b_{\kappa^{\prime\prime},\lambda^{\prime\prime}}
        \Big) \Big)\bigg).
\end{align}

If $d_1 = 1$ and all other $d_k$ are even, then the degree $5-D$ coefficient of the
Laurent series is given by
\begin{align}
    \label{eq:Gamma2FirstCase2}
&\sum\limits_{(K,I)\in\Lambda}
        \frac{b_{K,I}^{D-6}}
        {24 \prod\limits_{\substack{(k,i)\in\Theta\smallsetminus \\ \{ (K,I)\} }} (b_{K,I} - b_{k,i})}
        \bigg(  12b_{K,I}^3 - 44b_{K,I}^2 + 48b_{K,I} - 16
                \\ \nonumber &\quad\quad\quad +
                \sum\limits_{\substack{(\kappa,\lambda)\in\Theta\smallsetminus \\ \{ (K,I)\} }} b_{\kappa,\lambda} \Big(
                    - 16b_{K,I}^2 + 32 b_{K,I} - 16
                        - 4 b_{\kappa,\lambda} + 4 b_{K,I} b_{\kappa,\lambda}
                \\ \nonumber &\quad\quad\quad
                    + \sum\limits_{\substack{(\kappa^\prime,\lambda^\prime)\in\Theta\smallsetminus
                            \\ \{(K,I),(\kappa,\lambda)\}}} b_{\kappa^\prime,\lambda^\prime} \Big(
                        7 b_{K,I}
                        - 6
                        - 2 b_{\kappa,\lambda}
                        -  \sum\limits_{\substack{(\kappa^{\prime\prime},\lambda^{\prime\prime})\in\Theta\smallsetminus
                            \\ \{(K,I),(\kappa,\lambda),(\kappa^{\prime},\lambda^{\prime})\}}}
                            b_{\kappa^{\prime\prime},\lambda^{\prime\prime}}
        \Big) \Big)\bigg)
        \\ \nonumber &
        +
        \sum\limits_{\substack{(K,I)\in\Lambda\smallsetminus \\ \{ (1,0), (1,1) \} }}
        \frac{b_{K,I}^{D-6}}
            {4\prod\limits_{\substack{(k,i)\in\Theta\smallsetminus \\ \{ (K,I), (1,0), (1,1) \} }} (b_{K,I} - b_{k,i})}
            \Big( 3 b_{K,I} - b_{1,0} - b_{1,1} - 2
        \\ \nonumber &\quad\quad\quad
                - \sum\limits_{\substack{(\kappa,\lambda)\in\Theta\smallsetminus \\ \{ (K,I), (1,0), (1,1) \} }}
                b_{\kappa,\lambda} \Big).
\end{align}
\end{proposition}
\begin{proof}
First assume at least two $d_k$ or at least one $d_k > 1$ is odd. Again based on the observations of
Section~\ref{subsec:LaurentCases}, we need only consider the terms given in Equation~\eqref{eq:TermZ1}.
A term of degree $5-D$ can arise from the Cauchy product formula from these factors in one of seven ways:
\begin{enumerate}
\item   The degree $1$ term from the numerator, the degree $-1$ term from each factor of the
        denominator except three, say
        $(\kappa,\lambda), (\kappa^\prime,\lambda^\prime), (\kappa^{\prime\prime},\lambda^{\prime\prime}) \neq (K,I)$,
        and the degree $0$ term from the factors of the denominator corresponding to $(\kappa,\lambda)$,
        $(\kappa^\prime,\lambda^\prime)$, and $(\kappa^{\prime\prime}, \lambda^{\prime\prime})$;

\item   The degree $1$ term from the numerator, the degree $-1$ term from each factor of the denominator except
        two, say $(\kappa,\lambda), (\kappa^\prime,\lambda^\prime) \neq (K,I)$, the degree $1$ term from the factor
        of the denominator corresponding to $(\kappa,\lambda)$, and the degree $0$ term from the factor
        of the denominator corresponding to $(\kappa^\prime,\lambda^\prime)$;

\item   The degree $1$ term from the numerator, the degree $-1$ term from each factor of the denominator except
        one, say $(\kappa,\lambda) \neq (K,I)$, and the degree $2$ term from the factor
        of the denominator corresponding to $(\kappa,\lambda)$;

\item   The degree $2$ term from the numerator, the degree $-1$ term from each factor of the denominator except
        two, say $(\kappa,\lambda), (\kappa^\prime,\lambda^\prime) \neq (K,I)$, and the degree $0$ terms from the factors
        of the denominator corresponding to $(\kappa,\lambda)$ and $(\kappa^\prime,\lambda^\prime)$;

\item   The degree $2$ term from the numerator, the degree $-1$ term from each factor of the denominator except
        one, say $(\kappa,\lambda) \neq (K,I)$, and the degree $1$ term from the factor
        of the denominator corresponding to $(\kappa,\lambda)$;

\item   The degree $3$ term from the numerator, the degree $-1$ term from each factor of the denominator except
        one, say $(\kappa,\lambda) \neq (K,I)$, and the degree $0$ term from the factor
        of the denominator corresponding to $(\kappa,\lambda)$;

\item   The degree $4$ term from the numerator and the degree $-1$ term from each factor of the denominator.
\end{enumerate}

In the case of (1), we compute
\[
    \frac{-b_{K,I}^{D-6} b_{\kappa,\lambda}b_{\kappa^\prime,\lambda^\prime}b_{\kappa^{\prime\prime},\lambda^{\prime\prime}}}
        {4\prod\limits_{\substack{(k,i)\in\Theta\smallsetminus \\ \{ (K,I)\} }} (b_{K,I} - b_{k,i})}.
\]
For (2), we have
\[
    \frac{b_{K,I}^{D-6} (2 b_{K,I} - b_{\kappa,\lambda}) b_{\kappa,\lambda}b_{\kappa^\prime,\lambda^\prime}}
        {12\prod\limits_{\substack{(k,i)\in\Theta\smallsetminus \\ \{ (K,I)\} }} (b_{K,I} - b_{k,i})}.
\]
For (3),
\[
    \frac{-b_{K,I}^{D-5} (2 b_{K,I} - b_{\kappa,\lambda}) b_{\kappa,\lambda}}
        {12\prod\limits_{\substack{(k,i)\in\Theta\smallsetminus \\ \{ (K,I)\} }} (b_{K,I} - b_{k,i})}.
\]
For (4),
\[
    \frac{b_{K,I}^{D-6} (b_{K,I} - 2) b_{\kappa,\lambda}b_{\kappa^\prime,\lambda^\prime}}
        {4\prod\limits_{\substack{(k,i)\in\Theta\smallsetminus \\ \{ (K,I)\} }} (b_{K,I} - b_{k,i})}.
\]
For (5),
\[
    \frac{-b_{K,I}^{D-6} (b_{K,I} - 2)(2 b_{K,I} - b_{\kappa,\lambda}) b_{\kappa,\lambda}}
        {12\prod\limits_{\substack{(k,i)\in\Theta\smallsetminus \\ \{ (K,I)\} }} (b_{K,I} - b_{k,i})}.
\]
For (6),
\[
    \frac{-b_{K,I}^{D-6} (b_{K,I}^2 - 3b_{K,I} + 2) b_{\kappa,\lambda}}
        {3\prod\limits_{\substack{(k,i)\in\Theta\smallsetminus \\ \{ (K,I)\} }} (b_{K,I} - b_{k,i})}.
\]
Finally, for (7),
\[
    \frac{b_{K,I}^{D-6} (3b_{K,I}^3 - 11b_{K,I}^2 + 12b_{K,I} - 4) }
        {6\prod\limits_{\substack{(k,i)\in\Theta\smallsetminus \\ \{ (K,I)\} }} (b_{K,I} - b_{k,i})}.
\]
Summing over all $(K,I)$, $(\kappa,\lambda)$, $(\kappa^\prime,\lambda^\prime)$ and $(\kappa^{\prime\prime},\lambda^{\prime\prime})$, we divide
(1) and (4) by $3!$ and $2!$, respectively, to account for the same choices appearing more than once
in the sum. If all $d_k$ are even, each $H_{V,K,I,1}(\bs{b}_\Theta,t) = H_{V,K,I,-1}(\bs{b}_\Theta,t)$, indicating
the $\sigma_V$ prefactor. This completes the proof of Equation~\eqref{eq:Gamma2FirstCase1}.

When $d_1 = 1$ and all other $d_i$ are even, $\sigma_V = 1$, and $H_{V,K,I,-1}(\bs{b}_\Theta, t)$ is given in
Equation~\eqref{eq:Case2TermNot1}
and has pole order $4-D$. A term of degree $5-D$ can arise from the Cauchy product formula from these factors in one
of three ways:
\begin{enumerate}
\item   The degree $1$ term from the numerator,
        the degree $0$ term from each holomorphic factor of the denominator,
        the degree $-1$ term from each singular factor of the
        denominator except one, say $(\kappa,\lambda) \neq (K,I)$, and the degree $0$ term from
        the factor of the denominator corresponding to $(\kappa,\lambda)$;
\item   The degree $2$ term from the numerator, the degree $0$ term from each holomorphic factor of the denominator,
        and the degree $-1$ term from each singular factor of the denominator;
\item   The degree $1$ term from the numerator, the degree $0$ term from one holomorphic factor of the denominator
        the degree $1$ term from the other holomorphic factor of the denominator, and the degree $-1$ term from each
        singular factor of the denominator.
\end{enumerate}
In the case of (1), we compute
\[
    \frac{- b_{K,I}^{D-6} b_{\kappa,\lambda}}
        {4\prod\limits_{\substack{(k,i)\in\Theta\smallsetminus \\ \{ (K,I), (1,0), (1,1) \} }} (b_{K,I} - b_{k,i})},
    \quad\quad\quad (\kappa > 1).
\]
For (2),
\[
    \frac{b_{K,I}^{D-6}(b_{K,I} - 2)}
        {4\prod\limits_{\substack{(k,i)\in\Theta\smallsetminus \\ \{ (K,I), (1,0), (1,1) \} }} (b_{K,I} - b_{k,i})}.
\]
For (3), we have the two terms,
\[
    \frac{b_{K,I}^{D-6}(b_{K,I} - b_{1,0})}
        {4\prod\limits_{\substack{(k,i)\in\Theta\smallsetminus \\ \{ (K,I), (1,0), (1,1) \} }} (b_{K,I} - b_{k,i})}
\]
and
\[
    \frac{b_{K,I}^{D-6}(b_{K,I} - b_{1,1})}
        {4\prod\limits_{\substack{(k,i)\in\Theta\smallsetminus \\ \{ (K,I), (1,0), (1,1) \} }} (b_{K,I} - b_{k,i})}.
\]
Summing these and adding them to the contributions of $H_{V,K,I,1}(\bs{b}_\Theta,t)$ described in
Equation~\eqref{eq:Gamma2FirstCase1} completes the proof of Equation~\eqref{eq:Gamma2FirstCase2}.
\end{proof}


\section{The Coefficients of the Laurent Expansion in Terms of Schur polynomials}
\label{sec:LaurentSchur}

In this section, we use Propositions~\ref{prop:Gamma0First}, \ref{prop:Gamma1First} and \ref{prop:Gamma2First} to give
explicit formulas for the $\gamma_m$ in terms of Schur polynomials in the variables $a_{k,i}$.
First, let us briefly recall the definition of Schur polynomials for the convenience of the
reader.

Recall that if
$\rho = (\rho_1,\ldots,\rho_n) \in \Z^n$ is an integer partition, i.e. $\rho_1\geq\rho_2\geq\cdots\geq \rho_n\geq 0$,
then the \emph{alternant} associated to $\rho$ in the variables $\bs{x} = (x_1,\ldots,x_n)$ is defined by
\[
    A_\rho(\bs{x}) = \det\left( x_i^{ \rho_j } \right).
\]
The alternant is an alternating polynomial in the $x_i$ and hence divisible by the Vandermonde determinant
\[
    A_\delta(\bs{x}) = \prod\limits_{1\leq i < j\leq n} (x_i - x_j),
\]
where $\delta = (n-1,n-2,\ldots,1,0)$. The \emph{Schur polynomial} associated to $\rho$ in the variables $x_i$
is defined to be
\begin{equation}
\label{eq:SchurDef}
    s_\rho(\bs{x}) :=   \frac{ A_{\delta+\rho}( \bs{x} ) }{ A_\delta ( \bs{x} ) }.
\end{equation}

\begin{remark}
\label{rem:NonStandScur}
For simplicity, we will sometimes refer to the Schur polynomial associated to $\rho\in\Z^n$ that fail to
be partitions in the sense that $\rho_1 < \rho_2$. In these cases, we mean the polynomial (or
Laurent polynomial if $\rho_1 < 0$) defined in the same way by Equation~\eqref{eq:SchurDef}. Note that
the alternant $A_{\delta+\rho}(\bs{x})$ is still alternating so that $s_\rho(\bs{x})$ is a symmetric
(Laurent) polynomial; however, such a polynomial may be zero for nontrivial $\rho$.
\end{remark}

It is easy to see that the expressions in Equations~\eqref{eq:Gamma0First},
\eqref{eq:Gamma1FirstCase1}, \eqref{eq:Gamma1FirstCase2}, \eqref{eq:Gamma2FirstCase1},
and \eqref{eq:Gamma2FirstCase2} can be broken down into linear combinations
of sums of the form
\begin{align}
\label{eq:GenericSchurSumRSTU}
    &\Sigma_{R,S,T,U}(\bs{b}_\Theta)
    \\ \nonumber &\quad\quad =
    \sum\limits_{(K,I) \in \Lambda} \sum\limits_{\substack{(\kappa,\lambda)\in\Theta\smallsetminus \\ \{(K,I)\}}}
        \sum\limits_{\substack{(\kappa^\prime,\lambda^\prime)\in\Theta\smallsetminus \\ \{(K,I),(\kappa,\lambda)\}}}
        \sum\limits_{\substack{(\kappa^{\prime\prime},\lambda^{\prime\prime})\in\Theta\smallsetminus \\
                \{(K,I),(\kappa,\lambda),(\kappa^\prime,\lambda^\prime)\}}}
            \frac{b_{K,I}^{R} b_{\kappa,\lambda}^S b_{\kappa^\prime,\lambda^\prime}^T
                    b_{\kappa^{\prime\prime},\lambda^{\prime\prime}}^U}
                {\prod\limits_{\substack{(k,i)\in\Theta\smallsetminus \\ \{ (K,I)\} }}
                    (b_{K,I} - b_{k,i}) },
\end{align}
\begin{align}
\label{eq:GenericSchurSumRST}
    \Sigma_{R,S,T}(\bs{b}_\Theta)
        &:=
        \sum\limits_{(K,I) \in \Lambda} \sum\limits_{\substack{(\kappa,\lambda)\in\Theta\smallsetminus \\ \{(K,I)\}}}
        \sum\limits_{\substack{(\kappa^\prime,\lambda^\prime)\in\Theta\smallsetminus \\ \{(K,I),(\kappa,\lambda)\}}}
            \frac{b_{K,I}^{R} b_{\kappa,\lambda}^S b_{\kappa^\prime,\lambda^\prime}^T }
                {\prod\limits_{\substack{(k,i)\in\Theta\smallsetminus \\ \{ (K,I)\} }}
                    (b_{K,I} - b_{k,i}) },
    \\
\label{eq:GenericSchurSumRS}
    \Sigma_{R,S}(\bs{b}_\Theta)
        &:=
        \sum\limits_{(K,I) \in \Lambda} \sum\limits_{\substack{(\kappa,\lambda)\in\Theta\smallsetminus \\ \{(K,I)\}}}
            \frac{b_{K,I}^{R} b_{\kappa,\lambda}^S }
                {\prod\limits_{\substack{(k,i)\in\Theta\smallsetminus \\ \{ (K,I)\} }}
                    (b_{K,I} - b_{k,i}) },
    \\
\label{eq:GenericSchurSumR}
    \Sigma_{R}(\bs{b}_\Theta)
        &:=
        \sum\limits_{(K,I) \in \Lambda}
            \frac{ b_{K,I}^{R} }
                {\prod\limits_{\substack{(k,i)\in\Theta\smallsetminus \\ \{ (K,I)\} }}
                    (b_{K,I} - b_{k,i}) },
\end{align}
for integers $R$, $S$, $T$, and $U$. Hence, we will first indicate how such sums can be
expressed in terms of Schur polynomials.

We give $\Theta$ the lexicographic ordering so that
$(k,i)\leq(k^\prime,i^\prime)$ if $k < k^\prime$ or $k = k^\prime$ and $i \leq i^\prime$. This gives a total
ordering on $\Theta$ and hence $\Lambda$, and we use $\iota(k,i)$ to denote the position of $(k,i)$
with respect to this ordering. That is, $\iota(k,i) = 1$ for the first element of $\Lambda$,
$\iota(k,i) = 2$ for the second, etc.

We define $P_S$ to be the power sum of degree $S$ in the $D$ variables $\bs{b}_\Theta = ( b_{k,i} : (k,i)\in\Theta)$,
i.e.
\[
    P_S :=  P_S(\bs{b}_\Theta) := \sum\limits_{(k,i)\in\Theta} b_{k,i}^S.
\]
Recall that $C=|\Lambda|$, and that $e$ denotes the number of $k$ such that $d_k$ is even.
For brevity, we will use $s_M := s_M(\bs{b}) :=  s_{M,C-2,C-3, \ldots, 1, 0}(\bs{b})$ to denote the Schur polynomial of this form
in the $C$ variables $\bs{b} := \big(b_{k,i} : (k,i)\in\Lambda\big)$.

\begin{lemma}
\label{lem:GeneralSchur}
Let $V = \bigoplus_{k=1}^r V_{d_k}$ be an $\SL_2$-representation with $V^{\SL_2} = \{0\}$.
Choose $b_{k,i} = 0$ for each $(k,i)\in\Theta$ such that $a_{k,i} = 0$, and assume
$b_{K,d_K - I} = -b_{K,I}$ for each $(K,I)\in\Lambda$. For $R, S, T, U\in\Z$, we have
\begin{align}
\label{eq:GeneralSchurRSTU}
    &\Sigma_{R,S,T,U}(\bs{b}_\Theta)
    \\ \nonumber & \quad\quad =
    \frac{1}{ 2 s_{\delta}(\bs{b}) }\Big(
        (2P_{S+T+U} - P_S P_{T+U} - P_T P_{S+U}
    \\ \nonumber &\quad\quad\quad\quad\quad
         - P_U P_{S+T} + P_S P_T P_U)s_{R - e - C}
    \\ \nonumber &\quad\quad\quad
        + (P_{T+U} - P_T P_U) s_{R + S - e - C}
        + (P_{S+U} - P_S P_U) s_{R + T - e - C}
    \\ \nonumber &\quad\quad\quad
        + (P_{S+T} - P_S P_T) s_{R + U - e - C}
    \\ \nonumber &\quad\quad\quad
        + 2 P_U s_{R + S + T - e - C}+ 2 P_T s_{R + S + U - e - C} + 2 P_S s_{R + T + U - e - C}
    \\ \nonumber &\quad\quad\quad
        - 6 s_{R + S + T + U - e - C}
    \Big)
\end{align}
\begin{align}
\label{eq:GeneralSchurRST}
    \Sigma_{R,S,T}(\bs{b}_\Theta)
    &=
    \frac{1}{ 2 s_{\delta}(\bs{b}) }\Big(
        (P_S P_T - P_{S+T})s_{R - e - C} - P_T s_{R + S - e - C}
    \\ \nonumber &\quad\quad\quad
        - P_S s_{R + T - e - C} + 2 s_{R + S + T - e - C} \Big),
\end{align}
\begin{align}
\label{eq:GeneralSchurRS}
    \Sigma_{R,S}(\bs{b}_\Theta)
    &=
    \frac{P_S s_{R - e - C} - s_{R + S - e - C}}{ 2 s_{\delta}(\bs{b}) },
    \quad\quad\quad\quad\mbox{and}
    \\ \label{eq:GeneralSchurR}
    \Sigma_R(\bs{b}_\Theta)
    &=
    \frac{ s_{R - e - C} }{ 2 s_{\delta}(\bs{b}) }.
\end{align}
\end{lemma}
Note that $a_{K,d_K - I} = -a_{K,I}$ so that the required relation holds in the limit $\bs{b}_\Theta\to\bs{a}_\Theta$.
Depending on the values of $R$, $S$, $T$, $U$ and $e$, it may be that the partitions appearing in
Equations~\eqref{eq:GeneralSchurRSTU}, \eqref{eq:GeneralSchurRST}, \eqref{eq:GeneralSchurRS}, and \eqref{eq:GeneralSchurR}
are not non-increasing in the first two entries so that the Schur polynomials are non-standard in the
sense described in Remark~\ref{rem:NonStandScur}.
\begin{proof}
We express
\begin{align*}
    \Sigma_{R,S,T,U}(\bs{b}_\Theta)
    &=
    \sum\limits_{(K,I) \in \Lambda}
        \frac{b_{K,I}^{R}
            \sum\limits_{\substack{(\kappa,\lambda),(\kappa^\prime,\lambda^\prime),
                (\kappa^{\prime\prime},\lambda^{\prime\prime})\in \\
                \Theta\smallsetminus \{(K,I)\},\,
                \mbox{\scriptsize distinct}}}
                    b_{\kappa,\lambda}^S b_{\kappa^\prime,\lambda^\prime}^T
                    b_{\kappa^{\prime\prime},\lambda^{\prime\prime}}^U }
                {2b_{K,I} (b_{K,I} - 0)^e \prod\limits_{\substack{(k,i)\in \\ \Lambda\smallsetminus \{(K,I)\} }}
                    (b_{K,I} - b_{k,i})(b_{K,I} + b_{k,i}) }
    \\&=
    \sum\limits_{(K,I) \in \Lambda}
        \frac{b_{K,I}^{R-e-1}
            \sum\limits_{\substack{(\kappa,\lambda),(\kappa^\prime,\lambda^\prime),
                (\kappa^{\prime\prime},\lambda^{\prime\prime})\in \\
                \Theta\smallsetminus \{(K,I)\},\,
                \mbox{\scriptsize distinct}}}
                    b_{\kappa,\lambda}^S b_{\kappa^\prime,\lambda^\prime}^T
                    b_{\kappa^{\prime\prime},\lambda^{\prime\prime}}^U }
                {2\prod\limits_{\substack{(k,i)\in \\ \Lambda\smallsetminus \{(K,I)\} }}
                    (b_{K,I}^2 - b_{k,i}^2) },
\end{align*}
where the fact that $b_{K,d_K - I} = -b_{K,I}$ implies that each factor $b_{K,I} - b_{K,d_K - I} = 2b_{K,I}$.
For brevity, define
\[
    \Psi_{S,T,U}(K,I)   :=
        \sum\limits_{\substack{(\kappa,\lambda),(\kappa^\prime,\lambda^\prime),
            (\kappa^{\prime\prime},\lambda^{\prime\prime})\in \\
            \Theta\smallsetminus \{(K,I)\},\,
            \mbox{\scriptsize distinct}}}
         b_{\kappa,\lambda}^S b_{\kappa^\prime,\lambda^\prime}^T
         b_{\kappa^{\prime\prime},\lambda^{\prime\prime}}^U,
\]
which we rewrite as
\begin{align}
    \nonumber
    \Psi_{S,T,U}(K,I)
    &=  (P_S - b_{K,I}^S)(P_T - b_{K,I}^T)(P_U - b_{K,I}^U)
            - (P_{S+T} - b_{K,I}^{S+T})(P_U - b_{K,I}^U)
        \\ \nonumber &\quad\quad
            - (P_{S+U} - b_{K,I}^{S+U})(P_T - b_{K,I}^T)
            - (P_{T+U} - b_{K,I}^{T+U})(P_S - b_{K,I}^S)
        \\ \nonumber &\quad\quad
            + 2(P_{S+T+U} - b_{K,I}^{S+T+U})
    \\ \label{eq:PsiRSTU}
    &=    2P_{S+T+U} - P_S P_{T+U} - P_T P_{S+U} - P_U P_{S+T}
            + P_S P_T P_U
        \\ \nonumber &\quad\quad
            - b_{K,I}^S P_T P_U - b_{K,I}^T P_S P_U - b_{K,I}^U P_S P_T
        \\ \nonumber &\quad\quad
            + 2 b_{K,I}^{S+T} P_U + 2 b_{K,I}^{S+U} P_T + 2 b_{K,I}^{T+U} P_S
        \\ \nonumber &\quad\quad
            + b_{K,I}^{S} P_{T+U} + b_{K,I}^{T} P_{S+U} + b_{K,I}^{U} P_{S+T}
            - 6 b_{K,I}^{S+T+U}.
\end{align}

Using the ordering of $\Theta$ described above, we can express
\begin{align}
    \label{eq:lastPhiRSTU}
    &\sum\limits_{(K,I) \in \Lambda}
        \frac{b_{K,I}^{R - e - 1} \Psi_{S,T,U}(K,I)}
            {2\prod\limits_{\substack{(k,i)\in \\ \Lambda\smallsetminus \{(K,I)\} }}
                (b_{K,I}^2 - b_{k,i}^2) }
    \\ \nonumber &=
    \frac{ \sum\limits_{(K,I) \in \Lambda} (-1)^{\iota(K,I)-1} b_{K,I}^{R - e - 1} \Psi_{S,T,U}(K,I)
            \prod\limits_{\substack{(k,i),(\kappa,\lambda)\in\Lambda\smallsetminus \{(K,I)\} \\ (k,i) < (\kappa,\lambda) }}
                    (b_{k,i}^2 - b_{\kappa,\lambda}^2)}
                {2\prod\limits_{\substack{(k,i),(\kappa,\lambda)\in\Lambda \\ (k,i) < (\kappa,\lambda) }}
                    (b_{k,i}^2 - b_{\kappa,\lambda}^2) }.
\end{align}
Expanding over the expression for $\Psi_{S,T,U}(K,I)$ in Equation~\eqref{eq:PsiRSTU}, this expression can
be written as a sum of similar expressions, which we consider simultaneously.

Recall that $C = |\Lambda|$. For any
integer $M$, we have
\begin{align*}
    &\frac{ \sum\limits_{(K,I) \in \Lambda} (-1)^{\iota(K,I)-1}
        b_{K,I}^M
        \prod\limits_{\substack{(k,i),(\kappa,\lambda)\in\Lambda\smallsetminus \{(K,I)\} \\ (k,i) < (\kappa,\lambda) }}
            (b_{k,i}^2 - b_{\kappa,\lambda}^2) }
            {\prod\limits_{\substack{(k,i),(\kappa,\lambda)\in\Lambda \\ (k,i) < (\kappa,\lambda) }}
                (b_{k,i}^2 - b_{\kappa,\lambda}^2) }
    \\&\quad\quad=
    \frac{ \det\left( b_{k,i}^{\rho_j + C - j} \right)_{\substack{(k,i)\in\Lambda \\ 1\leq j \leq C}} }
        {\prod\limits_{\substack{(k,i),(\kappa,\lambda)\in\Lambda \\ (k,i) < (\kappa,\lambda) }}
            (b_{k,i} + b_{\kappa,\lambda})(b_{k,i} - b_{\kappa,\lambda}) },
\end{align*}
where $\rho = (M-C+1, C-2, C-3, C-4, \ldots, 1, 0)$, and where the matrix
$\left(b_{k,i}^{\rho_j + n - j}\right)$ is interpreted as indexing rows by $(k,i)\in\Lambda$ (in terms
of the order described above) and columns by $j = 1,\ldots,C$. To see this last step, notice that the numerator
of the previous equation can be seen as the cofactor expansion of the determinant along the first column.
Hence,
\[
    \frac{ \det\left( b_{k,i}^{\rho_j + C - j} \right)_{\substack{(k,i)\in\Lambda \\ 1\leq j \leq C}} }
        {\prod\limits_{\substack{(k,i),(\kappa,\lambda)\in\Lambda \\ (k,i) < (\kappa,\lambda) }}
            (b_{k,i} + b_{\kappa,\lambda})(b_{k,i} - b_{\kappa,\lambda}) }
    =
    \frac{ s_{\rho}(\bs{b}) }
        {\prod\limits_{\substack{(k,i),(\kappa,\lambda)\in\Lambda \\ (k,i) < (\kappa,\lambda) }}
            (b_{k,i} + b_{\kappa,\lambda}) }
    =
    \frac{ s_{\rho}(\bs{b}) }
        { s_{\delta}(\bs{b}) }.
\]
Now, expanding Equation~\eqref{eq:lastPhiRSTU} using Equation~\eqref{eq:PsiRSTU} and applying this
observation to each term yields Equation~\eqref{eq:GeneralSchurRSTU}.
Equations~\eqref{eq:GeneralSchurRST}, \eqref{eq:GeneralSchurRS}, and \eqref{eq:GeneralSchurR}
are proved identically, replacing $\Psi_{S,T,U}(K,I)$ with
\begin{align*}
    \Psi_{S,T}(K,I)
    &:=     \sum\limits_{\substack{(\kappa,\lambda),(\kappa^\prime,\lambda^\prime)\in  \\
                \Theta\smallsetminus \{(K,I)\},\,
                    \mbox{\scriptsize distinct}}}
                b_{\kappa,\lambda}^S b_{\kappa^\prime,\lambda^\prime}^T
    \\&=    (P_S - b_{K,I}^S)(P_T - b_{K,I}^T) - (P_{S+T} - b_{K,I}^{S+T})
    \\&=    P_S P_T - P_{S+T} - b_{K,I}^S P_T - b_{K,I}^T P_S + 2 b_{K,I}^{S+T},
\\
    \Psi_{S}(K,I)
    &:=     \sum\limits_{\substack{(\kappa,\lambda) \in \Theta\smallsetminus \{(K,I)\}}}
                b_{\kappa,\lambda}^S
    \\&=    P_S - b_{K,I}^S,
\end{align*}
and $1$, respectively.
\end{proof}

\begin{remark}
\label{rem:PartialSchur}
Note that one may also consider the expressions in Equations~\eqref{eq:GeneralSchurRSTU},
\eqref{eq:GeneralSchurRST}, \eqref{eq:GeneralSchurRS}, and \eqref{eq:GeneralSchurR} as polynomials in the variables
$\big(b_{k,i} : (k,i)\in\Theta\big)$. The resulting expressions are \emph{partial Laurent-Schur
polynomials} as defined in \cite[Section 5]{CowieHerbigSeatonHerden}, where the two sets of
variables correspond to the $b_{k,i}$ with $(k,i)\in\Lambda$ and $(k,i)\notin\Lambda$.
\end{remark}

With this, we can now complete the computations of $\gamma_m$ for $m=0,1,2,3$ and hence the proof
of Theorem~\ref{thrm:Main}. We claim that as $\bs{b}_\Theta\to\bs{a}_\Theta$, the expressions given by
Propositions~\ref{prop:Gamma0First}, \ref{prop:Gamma1First}, and \ref{prop:Gamma2First}
tend to $\gamma_0$, $\gamma_1$, and $\gamma_2$, respectively. This can be seen by noting that the integrands in the definitions of the Laurent
coefficients are continuous on the circle and hence can be bounded as $\bs{b}_\Theta\to\bs{a}_\Theta$, so the
Dominated Convergence Theorem allows one to exchange the limit with the integral. See
\cite[end of Section 5-2]{HerbigSeaton}, where this argument is given in detail in a very similar case.

We first consider $\gamma_0$.

\begin{theorem}
\label{thrm:Gamma0Schur}
Let $V = \bigoplus_{k=1}^r V_{d_k}$ be an $\SL_2$-representation with $V^{\SL_2} = \{0\}$,
and assume $V$ is not isomorphic to $2V_1$, nor $V_d$ for $d \leq 4$.
Let $\bs{a} := \big(a_{k,i} : (k,i)\in\Lambda\big)$. The degree $3-D$ coefficient $\gamma_0$ of
the Laurent series of $\Hilb_V(t)$ is given by
\begin{equation}
\label{eq:Gamma0Schur}
    \gamma_0
    =
    \frac{\sigma_V s_{\rho}(\bs{a}) }
        {s_{\delta}(\bs{a})}.
\end{equation}
where $\rho = (C-3,C-3,C-3,C-4,\ldots,1,0)$ and $\delta = (C-1,C-2,\ldots,1,0)$.
\end{theorem}
Note that if $C  = 2$, e.g. if $V = V_1\oplus V_2$ or $V = 2V_2$, then the partition
$(C-3,C-3,C-3,C-4,\ldots, 1, 0) = (-1,-1)$ and hence corresponds to a Laurent-Schur
polynomial; see Remark~\ref{rem:NonStandScur}. In all other cases under consideration,
the partition appearing in Equation~\eqref{eq:Gamma0Schur} is a standard integer
partition.
\begin{proof}
As above, we let
$\bs{b} := \big(b_{k,i} : (k,i)\in\Lambda\big)$ and
$\bs{b}_\Theta = ( b_{k,i} : (k,i)\in\Theta)$
where the $b_{k,i}$ denote real parameters and use the shorthand
$s_M := s_{M,C-2,C-3, \ldots, 1, 0}(\bs{b})$.
Using Equations~\eqref{eq:GeneralSchurRS} and \eqref{eq:GeneralSchurR},
we rewrite Equation~\eqref{eq:Gamma0First} as
\begin{align}
    \label{eq:Gamma0SchurStep}
    &\sigma_V \Big( \Sigma_{D-3}(\bs{b}) - 2 \Sigma_{D-4}(\bs{b}) - \Sigma_{D-4,1}(\bs{b}) \Big)
    \\ \nonumber &\quad\quad=
    \frac{\sigma_V}{2 s_\delta(\bs{b})} \Big(
        s_{D-e-C-3}(\bs{b})
        -
        2 s_{D-e-C-4}(\bs{b})
    \\ \nonumber
    &\quad\quad\quad\quad\quad
        -
        P_1(\bs{b}_\Theta) s_{D-e-C-4}(\bs{b})
        + s_{D-e-C-3}(\bs{b})
    \Big).
\end{align}
Recall that
\[
    s_\delta(\bs{b})
    =   \prod\limits_{\substack{(k,i),(\kappa,\lambda)\in\Lambda \\ (k,i) < (\kappa,\lambda) }}
            (b_{k,i} + b_{\kappa,\lambda}),
\]
where the order $<$ is that described before Lemma~\ref{lem:GeneralSchur}. As $a_{k,i} > 0$
for each $(k,i)\in\Lambda$, $s_\delta(\bs{a})\neq 0$. Hence, the limit as $\bs{b}\to\bs{a}$
of the expression in Equation~\eqref{eq:Gamma0SchurStep} is continuous at $\bs{b} = \bs{a}$.
Moreover, as the elements of $\bs{a}_\Theta$ are either zero or come in positive and negative
pairs, $P_1(\bs{a}_{\Theta}) = 0$. Noting that $D-e-C$ is the number of negative elements of $\Theta$, which is equal to the number $C$ of elements of $\Lambda$, yields
\[
    \gamma_0
    =
    \frac{\sigma_V}{s_\delta(\bs{a})} \Big(
        s_{C-3}(\bs{a}) - s_{C-4}(\bs{a})
    \Big).
\]
However, note that $s_{C-3}(\bs{a})$ is defined in Equation~\eqref{eq:SchurDef}
in terms of the alternant associated to $\delta+(C-3,C-2,C-3,\ldots,1,0) = (2C-4, 2C-4, 2C-6,\ldots,2,0)$.
Provided $C\geq 2$, which is true for all cases under consideration, the repetition of $2C-4$ implies that
$s_{C-3}(\bs{a}) = 0$. Now $s_{C-4}(\bs{a})$ is defined in terms of the
alternant associated to $\delta+(C-4,C-2,C-3,\ldots,1,0) = (2C-5, 2C-4, 2C-6,\ldots,2,0)$, which is not
in standard form. Switching the first two entries, we have
$(2C-4, 2C-5, 2C-6,\ldots,2,0) = \delta+(C-3, C-3, C-3, C-4,\ldots,1,0)$. Hence,
\begin{equation}
\label{eq:SchurC3}
    s_{C-4}(\bs{a})
    =
    - s_{C-3,C-3,C-3,C-4,\ldots,1,0}(\bs{a}),
\end{equation}
completing the proof.
\end{proof}

We now turn to the computation of $\gamma_1$.

\begin{theorem}
\label{thrm:Gamma1Schur}
Let $V = \bigoplus_{k=1}^r V_{d_k}$ be an $\SL_2$-representation with $V^{\SL_2} = \{0\}$,
and assume $V$ is not isomorphic to $2V_1$ nor $V_d$ for $d \leq 4$.
Let $\bs{a} := \big(a_{k,i} : (k,i)\in\Lambda\big)$. Then
\begin{equation}
\label{eq:Gamma1Schur}
    \gamma_1
    =
    \frac{3 \sigma_V s_{\rho}(\bs{a}) }{2 s_{\delta}(\bs{a}) }
    =
    \frac{3\gamma_0}{2}
\end{equation}
where $\rho = (C-3,C-3,C-3,C-4,\ldots,1,0)$ and $\delta = \delta_C =(C-1,C-2,\ldots,1,0)$.
\end{theorem}
\begin{proof}
We continue to let $\bs{b} := \big(b_{k,i} : (k,i)\in\Lambda\big)$,
$\bs{b}_\Theta = (b_{k,i} : (k,i)\in\Theta)$,
and $s_M := s_{M,C-2,C-3, \ldots, 1, 0}(\bs{b})$.
We first assume all $d_k$ are even, at least two $d_k$ are odd, or at least
one odd $d_k > 1$. Using Equations~\eqref{eq:GeneralSchurRST},
\eqref{eq:GeneralSchurRS}, and \eqref{eq:GeneralSchurR},
Equation~\eqref{eq:Gamma1FirstCase1} is equal to
\begin{align*}
    &\sigma_V\Big(
        \frac{2}{3} \Sigma_{D-3}(\bs{b}) - 2 \Sigma_{D-4}(\bs{b}) + \frac{4}{3} \Sigma_{D-5}(\bs{b})
            + \frac{1}{6}\Sigma_{D-5,2}(\bs{b}) - \frac{5}{6}\Sigma_{D-4,1}(\bs{b})
    \\&\quad\quad
            + \Sigma_{D-5,1}(\bs{b}) + \frac{1}{2}\Sigma_{D-5,1,1}(\bs{b}) \Big)
    \\&\quad\quad=
    \frac{\sigma_V}{ 2 s_{\delta}(\bs{b}) }\Big(
        \frac{2}{3} s_{D - e - C - 3}(\bs{b})
            - 2 s_{D - e - C - 4}(\bs{b})
            + \frac{4}{3} s_{D - e - C - 5}(\bs{b})
    \\&\quad\quad\quad
            + \frac{1}{6}\big(P_2(\bs{b}_\Theta) s_{D - e - C - 5}(\bs{b})
                - s_{D - e - C - 3}(\bs{b}) \big)
    \\&\quad\quad\quad
            - \frac{5}{6}\big(P_1(\bs{b}_\Theta) s_{D - e - C - 4}(\bs{b})
                - s_{D - e - C - 3}(\bs{b}) \big)
    \\&\quad\quad\quad
            + \big(P_1(\bs{b}_\Theta) s_{D - e - C - 5}(\bs{b})
                - s_{D - e - C - 4}(\bs{b}) \big)
    \\&\quad\quad\quad
            + \frac{1}{2}\Big( \big(P_1(\bs{b}_\Theta)^2 - P_{2} (\bs{b}_\Theta) \big)
                s_{D - e - C- 5}(\bs{b})
                - P_1(\bs{b}_\Theta) s_{D - e - C - 4}(\bs{b})
    \\&\quad\quad\quad
                - P_1 s_{D - e - C - 4} + 2 s_{D - e - C - 3}(\bs{b})
    \Big) \Big).
\end{align*}
Then using the fact that $D-e-C = C$ noted in the proof of Theorem~\ref{thrm:Gamma0Schur},
this is equal to
\begin{align*}
    &\frac{\sigma_V}{ 2 s_{\delta}(\bs{b}) }\Big(
    \frac{2}{3} s_{C - 3}(\bs{b})
            - 2 s_{C - 4}(\bs{b})
            + \frac{4}{3} s_{C - 5}(\bs{b})
            + \frac{1}{6}\big(P_2(\bs{b}_\Theta) s_{C - 5}(\bs{b})
                - s_{C - 3}(\bs{b}) \big)
    \\&\quad\quad\quad
            - \frac{5}{6}\big(P_1(\bs{b}_\Theta) s_{C - 4}(\bs{b})
                - s_{C - 3}(\bs{b}) \big)
            + \big(P_1(\bs{b}_\Theta) s_{C - 5}(\bs{b})
                - s_{C - 4}(\bs{b}) \big)
    \\&\quad\quad\quad
            + \frac{1}{2}\Big( \big(P_1(\bs{b}_\Theta)^2 - P_2 (\bs{b}_\Theta) \big)
                s_{C - 5}(\bs{b})
                - P_1(\bs{b}_\Theta) s_{C - 4}(\bs{b})
                - P_1 s_{C - 4} + 2 s_{C - 3}(\bs{b})
    \Big)\Big)
    \\&\quad\quad=
    \frac{\sigma_V}{12 s_{\delta}(\bs{b}) }\Big(
        14 s_{C - 3}(\bs{b})
        - \big( 11 P_1(\bs{b}_\Theta) + 18 \big) s_{C - 4}(\bs{b})
    \\&\quad\quad\quad
        + \big( -2 P_2(\bs{b}_\Theta) + 3 P_1(\bs{b}_\Theta)^2 + 6 P_1(\bs{b}_\Theta) + 8 \big)
            s_{C - 5}(\bs{b}) \Big).
\end{align*}
As in the proof of Theorem~\ref{thrm:Gamma0Schur}, $s_{\delta}(\bs{a}) \neq 0$ so that this expression
is continuous at $\bs{b} = \bs{a}$. Similarly, $P_1(\bs{a}_\Theta) = 0$, the non-standard Schur polynomial
$s_{C - 3}(\bs{a}) = 0$, and $s_{C-4}(\bs{a}) = - s_{C-3,C-3,C-3,C-4,\ldots,1,0}(\bs{a})$.
The non-standard Schur polynomial $s_{C - 5}(\bs{a})$ is defined in terms of the alternant associated to
$\delta+(C-5,C-2,C-3,\ldots,1,0) = (2C-6,2C-4,2C-6,\ldots,2,0)$ and hence vanishes.
This completes the proof in this case.

Now assume $d_1 = 1$ and all other $d_k$ are even.
We need only deal with the additional sum in Equation~\eqref{eq:Gamma1FirstCase2},
\begin{equation}
\label{eq:Case2SchurExtraTerms}
    \sum\limits_{\substack{(K,I)\in\Lambda \\ K \neq 1}}
        \frac{ b_{K,I}^{D-5} }
            { 2 \prod\limits_{\substack{(k,i)\in\Theta \\ (k,i)\neq(K,I) \\ k\neq 1}}
                (b_{K,I}-b_{k,i})}.
\end{equation}
Define $\Theta_1 = \{(k,i)\in\Theta : k\neq 1\}$, $\Lambda_1 = \{(k,i)\in\Lambda : k\neq 1\}$,
$\bs{a}_1 := \big(a_{k,i} : (k,i)\in\Lambda_1 \big)$,
$\bs{b}_1 := \big(b_{k,i} : (k,i)\in\Lambda_1 \big)$, and $\bs{b}_{\Theta_1} := \big(b_{k,i} : (k,i)\in\Theta_1 \big)$.
Note that we can treat $\Theta_1$ and $\Lambda_1$ as associated to the representation $\bigoplus_{k=2}^r V_{d_k}$,
which has dimension $D-2$; the value of $e$ is unchanged, and $\Lambda_1$ has cardinality $C - 1$. Then we can
rewrite Equation~\eqref{eq:Case2SchurExtraTerms} as
\[
    \frac{1}{2}\sum\limits_{(K,I)\in\Lambda_1}
        \frac{ b_{K,I}^{D-5} }
            { \prod\limits_{\substack{(k,i)\in\Theta_1 \\ (k,i)\neq(K,I)}}
                (b_{K,I}-b_{k,i})}
    =
    \frac{1}{2}\Sigma_{D-5}(\bs{b}_{\Theta_1}).
\]
Applying Equation~\eqref{eq:GeneralSchurR} and recalling that $D - e - C = C$, this is equal to
\[
    \frac{ s_{C-4, C-3,C-4,\ldots, 1, 0}(\bs{b}_1) }{ 4 s_{\delta}(\bs{b}_1) }.
\]
We again note that $s_{\delta}(\bs{a}_1) \neq 0$ so that this function is continuous at
$\bs{b}_1 = \bs{a}_1$. The Schur polynomial associated to the non-standard partition
$(C-4, C-3,C-4,\ldots, 1, 0)$ is defined by the alternant associated to
$\delta_{C-1}+(C-4, C-3,C-4,\ldots, 1, 0) = (2C-6, 2C-6, 2C-8,\ldots,2,0)$ and hence
vanishes, completing the proof. That this result is also true for $V_1+V_2$ can be verified by
direct computation; see Table~\ref{tab:Exceptions}.
\end{proof}

The computations of $\gamma_0$ and $\gamma_1$ yield a value for the $a$-invariant
$a(\C[V]^{\SL_2})$.

\begin{corollary}
\label{cor:aInvar}
Let $V = \bigoplus_{k=1}^r V_{d_k}$ be an $\SL_2$-representation with $V^{\SL_2} = \{0\}$,
and assume $V$ is not isomorphic to $2V_1$, nor $V_d$ for $d \leq 4$. Then
$a(\C[V]^{\SL_2}) = -D$.
\end{corollary}
\begin{proof}
Note that $\C[V]^{\SL_2}$ is Gorenstein by \cite[Corollary 1.9]{HochsterRoberts},
which implies that $ - 2\gamma_1/\gamma_0 = a(\C[V]^{\SL_2}) + \dim(\C[V]^{\SL_2})$, see \cite[Equation (3.32)]{PopovVinberg}.
All representations except $2V_1$ and $V_d$ for $d \leq 4$  are $1$-large by
\cite[Theorem 3.4]{HerbigSchwarz} implying by
\cite[Remark (9.2)(3)]{GWSlifting} that $\dim(\C[V]^{\SL_2}) = D - 3$.
This completes the proof.
\end{proof}

\begin{remark}
\label{rem:KnopLittelmanGam1Superfluous}
As noted in the Introduction, this computation of the $a$-invariant agrees with that given
in \cite[Satz 1]{KnopLittelmann}. Note that Knop and Littelmann's computation of $a(\C[V]^{\SL_2})$ along with our computation of $\gamma_0$ render the computation
of $\gamma_1$ superfluous; we could have instead reversed the logic of
Corollary~\ref{cor:aInvar} to conclude that $\gamma_1 = 3\gamma_0/2$.
\end{remark}

Next, we compute $\gamma_2$ in the following.

\begin{theorem}
\label{thrm:Gamma2Schur}
Let $V = \bigoplus_{k=1}^r V_{d_k}$ be an $\SL_2$-representation with $V^{\SL_2} = \{0\}$,
and assume $V$ is not isomorphic to $V_d$ for $d=1,2,3,4,5,6,8$, $2V_1$, $V_1+V_2$, $V_1+V_3$, $V_1+V_4$,
$2V_2$, $V_2+V_3$, $V_2+V_4$, $2V_3$, nor $2V_4$.
Let $\bs{a} := \big(a_{k,i} : (k,i)\in\Lambda\big)$.

If all $d_k$ are even, at least two $d_k$ are odd, or at least one odd $d_k > 1$, then
\begin{equation}
\label{eq:Gamma2SchurCase1}
    \gamma_2
    =
    \sigma_V\,
    \frac{42 s_{\rho}(\bs{a}) + s_{\rho^\prime}(\bs{a})
            \big( P_2(\bs{a})  - 8\big) }
                {24 s_\delta(\bs{a})}.
\end{equation}
where $\rho = (C-3,C-3,C-3,C-4,\ldots,1,0)$, $\rho^\prime = (C-3, C-4, C-4, C-4, C-5,\ldots,1,0)$,
and $\delta = \delta_C= (C-1,C-2,\ldots,1,0)$.

If $d_1 = 1$ and all other $d_k$ are even, then
\begin{equation}
\label{eq:Gamma2SchurCase2}
    \gamma_2
    =
    \frac{42 s_{\rho}(\bs{a}) + s_{\rho^\prime}(\bs{a})
            \big( P_2(\bs{a})  - 8\big) }
                {24 s_\delta(\bs{a})}
        + \frac{s_{C - 4, C - 4, C - 4, C - 5, \ldots, 1, 0}(\bs{a_1}) }
            {4 s_{C - 2, C - 3, C - 4 \ldots, 1, 0}(\bs{a_1}) }
\end{equation}
where $\bs{a}_1$ denotes $\bs{a}$ with the entry $a_{1,1}$ removed.
\end{theorem}
\begin{proof}
As above, $\bs{b} := \big(b_{k,i} : (k,i)\in\Lambda\big)$,
$\bs{b}_\Theta = (b_{k,i} : (k,i)\in\Theta)$,
and $s_M := s_{M,C-2,C-3, \ldots, 1, 0}(\bs{b})$.
We first assume all $d_k$ are even, at least two $d_k$ are odd, or at least
one odd $d_k > 1$.
Rewriting Equation~\eqref{eq:Gamma2FirstCase1} using
Equations~\eqref{eq:GeneralSchurRSTU}, \eqref{eq:GeneralSchurRST},
\eqref{eq:GeneralSchurRS}, \eqref{eq:GeneralSchurR} yields
\begin{align*}
            &\frac{\sigma_V}{48 s_\delta(\bs{b})}\Big(
                12 s_{D-3-e-C}(\bs{b}) - 44 s_{D-4-e-C}(\bs{b}) + 48 s_{D-5-e-C}(\bs{b})- 16 s_{D-6-e-C}(\bs{b})
            \\ &\quad\quad
                - 16 \big(P_1(\bs{b}_\Theta) s_{D-4-e-C}(\bs{b})
                - s_{D-4+1-e-C}(\bs{b})\big)
            \\ &\quad\quad
                + 32 \big(P_1(\bs{b}_\Theta) s_{D-5-e-C}(\bs{b}) - s_{D-5+1-e-C}(\bs{b})\big)
            \\ &\quad\quad
                - 16 \big(P_1(\bs{b}_\Theta) s_{D-6-e-C}(\bs{b}) - s_{D-6+1-e-C}(\bs{b})\big)
            \\ &\quad\quad
                - 4 \big(P_2(\bs{b}_\Theta) s_{D-6-e-C}(\bs{b}) - s_{D-6+2-e-C}(\bs{b})\big)
            \\ &\quad\quad
                + 4 \big(P_2(\bs{b}_\Theta) s_{D-5-e-C}(\bs{b}) - s_{D-5+2-e-C}(\bs{b})\big)
            \\ &\quad\quad
                + 7 \big( (P_1(\bs{b}_\Theta) P_1(\bs{b}_\Theta)
                    - P_{2}(\bs{b}_\Theta)) s_{D-5 - e - C}(\bs{b})
                    - P_1(\bs{b}_\Theta) s_{D - 5 + 1 - e - C}(\bs{b})
            \\ &\quad\quad
                    - P_1(\bs{b}_\Theta) s_{D - 5 + 1 - e - C}(\bs{b})
                    + 2 s_{D - 5 + 1 + 1 - e - C}(\bs{b})\big)
            \\ &\quad\quad
                - 6 \big( (P_1(\bs{b}_\Theta) P_1(\bs{b}_\Theta)
                - P_2(\bs{b}_\Theta))s_{D - 6 - e - C}(\bs{b})
                    - P_1(\bs{b}_\Theta) s_{D - 6 + 1 - e - C}(\bs{b})
            \\ &\quad\quad
                    - P_1(\bs{b}_\Theta) s_{D - 6 + 1 - e - C}(\bs{b})
                    + 2 s_{D - 6 + 1 + 1 - e - C}(\bs{b}) \big)
            \\ &\quad\quad
                - 2 \big( (P_2(\bs{b}_\Theta) P_1(\bs{b}_\Theta)
                    - P_3(\bs{b}_\Theta))s_{D - 6 - e - C}(\bs{b})
                    - P_1(\bs{b}_\Theta) s_{D - 6 + 2 - e - C}(\bs{b})
            \\ &\quad\quad
                    - P_2(\bs{b}_\Theta) s_{D - 6 + 1 - e - C}(\bs{b})
                    + 2 s_{D - 6 + 2 + 1 - e - C}(\bs{b}) \big)
                - \big((2P_3(\bs{b}_\Theta) - P_1(\bs{b}_\Theta) P_2(\bs{b}_\Theta)
            \\ &\quad\quad
                    - P_1(\bs{b}_\Theta) P_2(\bs{b}_\Theta)
                    - P_1(\bs{b}_\Theta) P_2(\bs{b}_\Theta)
                    + P_1(\bs{b}_\Theta) P_1(\bs{b}_\Theta) P_1(\bs{b}_\Theta))
                        s_{D - 6 - e - C}(\bs{b})
            \\ &\quad\quad
                    + (P_2(\bs{b}_\Theta) - P_1(\bs{b}_\Theta) P_1(\bs{b}_\Theta))
                        s_{D - 6 + 1 - e - C}(\bs{b})
            \\ &\quad\quad
                    + (P_2(\bs{b}_\Theta) - P_1(\bs{b}_\Theta) P_1(\bs{b}_\Theta))
                        s_{D - 6 + 1 - e - C}(\bs{b})
            \\ &\quad\quad
                    + (P_2(\bs{b}_\Theta) - P_1(\bs{b}_\Theta) P_1(\bs{b}_\Theta))
                        s_{D - 6 + 1 - e - C}(\bs{b})
            \\ &\quad\quad
                        + 2 P_1(\bs{b}_\Theta) s_{D - 6 + 1 + 1 - e - C}(\bs{b})
                    + 2 P_1(\bs{b}_\Theta) s_{D - 6 + 1 + 1 - e - C}(\bs{b})
            \\ &\quad\quad
                    + 2 P_1(\bs{b}_\Theta) s_{D - 6 + 1 + 1 - e - C}(\bs{b})
                    - 6 s_{D - 6 + 1 + 1 + 1 - e - C}(\bs{b})
            \big) \Big).
\end{align*}
Applying $D - e - C = C$ and $P_r(\bs{a}_{\Theta}) = 0$ for $r$ odd, this is equal to
\begin{align*}
    &\frac{\sigma_V}{24 s_\delta(\bs{b})}\Big(
    20 s_{C-3}(\bs{b}) - 42 s_{C-4}(\bs{b}) + 32 s_{C-5}(\bs{b})
            - 8 s_{C-6}(\bs{b})
    \\ &\quad\quad\quad
            + P_2(\bs{b}_\Theta) s_{C-6}(\bs{b}) - 2 P_2(\bs{b}_\Theta) s_{C-5}(\bs{b})
            \Big),
\end{align*}
which, as $s_{C-3}(\bs{a}) = 0$ and $s_{C-5}(\bs{a}) = 0$, is equal to
\[
    \sigma_V \frac{- 42 s_{C-4}(\bs{b}) + s_{C-6}(\bs{b})
                    \big( P_2(\bs{b}_\Theta)  - 8\big) }
                {24 s_\delta(\bs{b})}.
\]
Rewriting the non-standard Schur polynomial
$s_{C-6}(\bs{b}) = s_{C-6,C-2,C-3,\ldots,1,0}(\bs{b})$
in standard form yields
$s_{C-6}(\bs{b}) = s_{C-3, C-4, C-4, C-4, C-5,\ldots,1,0}(\bs{b})$.
Applying this as well as Equation~\eqref{eq:SchurC3} completes the proof of
Equation~\eqref{eq:Gamma2SchurCase1}.

Now suppose $d_1 = 1$ and $d_i$ is even for $i > 1$. As in the proof of
Theorem~\ref{thrm:Gamma1Schur}, we use the notation $\Theta_1 = \{(k,i)\in\Theta : k\neq 1\}$,
$\Lambda_1 = \{(k,i)\in\Lambda : k\neq 1\}$, and
$\bs{b}_1 := \big(b_{k,i} : (k,i)\in\Lambda_1 \big)$, etc.; then
$|\Theta_1| = D - 2$ and $|\Lambda_1| = C - 1$.
Applying Equations~\eqref{eq:GeneralSchurRS} and \eqref{eq:GeneralSchurR}
the last sum over $(K,I)$ in Equation~\eqref{eq:Gamma2FirstCase2}
can be written
\begin{align*}
    \\&
    \frac{1}{8 s_{\delta_{C-1}}(\bs{b_1}) } \Big(
        3 s_{D - 5 - e - C+1,C-3,\ldots,1,0}(\bs{b_1})
            - (b_{0,1} \! +\! b_{1,1}\! +\! 2)s_{D - 6 - e - C+1,C-3,\ldots,1,0}(\bs{b_1})
    \\&\quad\quad\quad
            - P_1(\bs{b_1}) s_{D - 6 - e - C+1,C-3,\ldots,1,0}(\bs{b_1})
            + s_{D - 5 - e - C+1,C-3,\ldots,1,0}(\bs{b_1})
        \Big)
    \\&=
    \frac{1}{8 s_{\delta_{C-1}}(\bs{b_1}) } \Big(
    4 s_{C - 4,C-3,\ldots,1,0}(\bs{b_1}) - (b_{0,1} + b_{1,1} + 2)
            s_{C - 5,C-3,\ldots,1,0}(\bs{b_1})
    \\&\quad\quad\quad
            - P_1(\bs{b_1}) s_{C - 5,C-3,\ldots,1,0}(\bs{b_1})
        \Big).
\end{align*}
Now, $s_{C - 5,C-3,C-4,\ldots,1,0}(\bs{b_1})$ is associated to the alternant
$\delta_{C-1} + (C-5, C-3, C-4, \ldots, 1, 0) = (2C - 7, 2C - 6, 2C - 8, 2C - 10, \ldots, 2, 0)$.
Permuting the first two entries yields
$(2C - 6, 2C - 7, 2C - 8, 2C - 10, \ldots, 2, 0) = \delta_{C-1} + (C - 4, C - 4, C - 4, C - 5, \ldots, 1, 0)$ so that
$s_{C-5,C-3,C-4,\ldots,1,0}(\bs{b_1}) = -s_{C-4,C-4,C-4,C-5,\ldots,1,0}(\bs{b_1})$.
Similarly, $s_{C-4,C-3,C-4,\ldots,1,0}(\bs{b_1})$ is associated to the alternant
$\delta_{C-1} + (C-4, C-3, C-4, \ldots, 1, 0) = (2C - 6, 2C - 6, 2C - 8, 2C - 10, \ldots, 2, 0)$,
which vanishes. Then noting that $P_1(\bs{b_1}) = 0$ yields
\[
    \frac{ (b_{0,1} + b_{1,1} + 2)s_{C - 4, C - 4, C - 4, C - 5, \ldots, 1, 0}(\bs{b_1}) }
        {8 s_{\delta_{C-1}}(\bs{b_1}) }.
\]
Taking the limit $\bs{b}\to\bs{a}$ and recalling that $a_{0,1} = -1$ and $a_{1,1} = 1$ yields
\[
    \frac{s_{C - 4, C - 4, C - 4, C - 5, \ldots, 1, 0}(\bs{a_1}) }{4 s_{\delta_{C-1}}(\bs{a_1}) }
    =
    \frac{s_{C - 4, C - 4, C - 4, C - 5, \ldots, 1, 0}(\bs{a_1}) }
        {4 s_{C - 2, C - 3, C - 4 \ldots, 1, 0}(\bs{a_1}) },
\]
completing the proof. Note that this last term can be interpreted as $\gamma_0^\prime/8$,
where $\gamma_0^\prime$ is the first Laurent coefficient of the Hilbert series associated to
$\bigoplus_{k=2}^r V_{d_k}$, unless $\bigoplus_{k=2}^r V_{d_k}$ is one of the exceptions for $\gamma_0$.
\end{proof}

Finally, we can quickly determine $\gamma_3$.

\begin{corollary}
\label{cor:Gamma3Schur}
Let $V = \bigoplus_{k=1}^r V_{d_k}$ be an $\SL_2$-representation with $V^{\SL_2} = \{0\}$,
and assume $V$ is not isomorphic to $2V_1$, nor $V_d$ for $d \leq 4$. Then
\begin{equation}
\label{eq:Gamma3Schur}
    \gamma_3 = \frac{5(\gamma_2 - \gamma_0)}{2}.
\end{equation}
\end{corollary}
\begin{proof}
As explained in the proof of Corollary~\ref{cor:aInvar}, $\C[V]^{\SL_2}$
is Gorenstein and, for the cases under consideration, $\dim(\C[V]^{\SL_2}) = D - 3$
and $a(\C[V]^{\SL_2}) = -D$. Hence, in the language of \cite{HerbigHerdenSeaton2},
the degree of the Gorenstein algebra
$\C[V]^{\SL_2}$ is $-\big( a(\C[V]^{\SL_2}) + \dim \C[V]^{\SL_2} \big) = 3$.
Therefore, by \cite[Theorem 1, Equation (1.6)]{HerbigHerdenSeaton2}, we have that
$10\gamma_1 - 15\gamma_2 + 6\gamma_3 = 0$, from which the result follows.
\end{proof}


\section{An Algorithm to Compute the Hilbert Series}
\label{sec:Algorithm}

In this section, we describe an algorithm to compute $\Hilb_V(t)$ for an arbitrary representation
$V = \bigoplus_{k=1}^r V_{d_k}$ of $\SL_2$.
This algorithm is similar to those given in \cite[Section 4]{HerbigSeaton} and
\cite[Section 4]{CowieHerbigSeatonHerden} for circle actions. Note, however, that those
algorithms consider only the \emph{generic} cases with no degeneracies caused by repeated weights.
In the case of $\SL_2$-invariants, this hypothesis is very restrictive; as was explained in
the introduction, it implies that $r \leq 2$ and, when $r = 2$, the degrees $d_1$ and $d_2$
have opposite parities. Hence, we begin by presenting a partial fraction decomposition in
Section~\ref{subsec:PartFrac} that allows us to extend to the general case. Note that this
decomposition can be used to extend the algorithms of \cite{CowieHerbigSeatonHerden,HerbigHerdenSeaton}
to the degenerate cases as well.


\subsection{Partial Fraction Decomposition}
\label{subsec:PartFrac}

The main partial fraction decomposition we consider is the following.

\begin{proposition}
\label{prop:PartFracGeneral}
For $t\in\C$, distinct values $x_1, \ldots, x_n \in \C$, and positive integers $m_1,\ldots, m_n$, we have
\begin{equation}
\label{eq:PartFracGeneral}
    \prod\limits_{i=1}^n \frac{1}{(1 - t x_i)^{m_i} }
    =
    \sum\limits_{i=1}^n \sum\limits_{j=0}^{m_i-1} \frac{ G_{i,j}(x_1,\ldots,x_n)}{(1 - t x_i)^{m_i - j} },
\end{equation}
where
\[
    G_{i,j}(x_1,\ldots,x_n) = \frac{1}{j! (-x_i)^j }
        \frac{d^j}{dt^j}\left.\left( \prod\limits_{\substack{k=1 \\ k\neq i}}^n
            \frac{1}{(1 - t x_k)^{m_k} }\right)\right|_{t= \frac{1}{x_i} } .
\]
\end{proposition}
\begin{proof}
Consider
\[
    f(t)    :=  \prod\limits_{i=1}^n \frac{1}{(1 - t x_i)^{m_i} }
\]
as a function of $t$. Clearly, a partial fraction decomposition
\[
    f(t)    =   C(x_1,\ldots,x_n)
                    + \sum\limits_{i=1}^n \sum\limits_{j=0}^{m_i-1}
                        \frac{ G_{i,j}(x_1,\ldots,x_n) }{(1 - t x_i)^{m_i - j} }
\]
is possible. Observe that $C(x_1,\ldots,x_n) = \lim_{t\to\infty} f(t) = 0$, so we need only evaluate the
$G_{i,j}(x_1,\ldots,x_n)$. We have
\begin{align}
    \nonumber
    &\prod\limits_{\substack{k=1 \\ k\neq i}}^n \frac{1}{(1 - t x_k)^{m_k}}
        =      (1 - t x_i)^{m_i} f(t)
    \\ \nonumber &\quad
    = \sum\limits_{\ell=0}^{m_i-1} G_{i,\ell}(x_1,\ldots,x_n)(1 - t x_i)^\ell
            + \sum\limits_{\substack{k=1 \\ k\neq i}}^n \sum\limits_{\ell=0}^{m_k-1}
                \frac{G_{k,\ell}(x_1,\ldots,x_n) (1 - t x_i)^{m_i}}
                    {(1 - t x_k)^{m_k-\ell}}
    \\ \label{eq:PartFracL1}
    &\quad=      \sum\limits_{\ell=0}^{j-1} G_{i,\ell}(x_1,\ldots,x_n)(1 - t x_i)^\ell
                + G_{i,j}(x_1,\ldots,x_n)(1 - t x_i)^j
    \\ \label{eq:PartFracL2}
    &\quad
                + \sum\limits_{\ell=j+1}^{m_i-1} G_{i,\ell}(x_1,\ldots,x_n)(1 - t x_i)^\ell
    \\ \label{eq:PartFracL3}
    &\quad
                + \sum\limits_{\substack{k=1\\k\neq i}}^n \sum\limits_{\ell=0}^{m_k-1}
                    \frac{G_{k,\ell}(x_1,\ldots,x_n) (1 - t x_i)^{m_i} }
                        {(1 - t x_k)^{m_k-\ell}}.
\end{align}
The first sum in \eqref{eq:PartFracL1} clearly has degree at most $j-1$ as a polynomial in $t$,
while, noting that $j+1 \leq m_i$, the expression in \eqref{eq:PartFracL2} and \eqref{eq:PartFracL3}
evidently has a zero at $t = 1/x_i$ of multiplicity at least $j+1$. Hence,
\begin{align*}
    & \frac{d^j}{dt^j}\left( \prod\limits_{\substack{k=1 \\ k\neq i}}^n \frac{1}{(1 - t x_k)^{m_k}} \right)
    =
    j!(-x_i)^j G_{i,j}(x_1,\ldots,x_n)
    \\
        + & \frac{d^j}{dt^j}\left( \sum\limits_{\ell=j+1}^{m_i-1} G_{i,\ell}(x_1,\ldots,x_n)(1 - t x_i)^\ell \right.
    \left.
                + \sum\limits_{\substack{k=1\\k\neq i}}^n \sum\limits_{\ell=0}^{m_k-1}
                    \frac{G_{k,\ell}(x_1,\ldots,x_n) (1 - t x_i)^{m_i} }
                        {(1 - t x_k)^{m_k-\ell}} \right),
\end{align*}
where $t = 1/x_i$ is a zero of the expression on the second line so that
\[
    \frac{d^j}{dt^j}\left.\left( \prod\limits_{\substack{k=1 \\ k\neq i}}^n \frac{1}{(1 - t x_k)^{m_k}} \right)
        \right|_{t = \frac{1}{x_i} }
    =
    j!(-x_i)^j G_{i,j}(x_1,\ldots,x_n).
    \qedhere
\]
\end{proof}

We observe two interesting special cases, beginning with the case where each $m_i = 1$.

\begin{corollary}
\label{cor:PartFracMi1}
For $t\in\C$ and distinct $x_1, \ldots, x_n\in\C$, we have
\[
    \prod\limits_{i=1}^n \frac{1}{1 - t x_i}
    =
    \sum\limits_{i=1}^n \frac{1}{(1 - t x_i)
        \prod\limits_{\substack{k=1 \\ k\neq i}}^n (1 - x_k/x_i)}.
\]
\end{corollary}

Restricting to the case $n = 2$ yields the following pleasing formula that we have come to
refer to as the \emph{Yin-Yang formula}.

\begin{corollary}
\label{cor:PartFracYinYang}
For $t\in\C$, $x, y \in \C$ distinct, and positive integers $\alpha$ and $\beta$, we have
\begin{equation*}
    \frac{1}{(1 - tx)^\alpha} \frac{1}{(1 - ty)^\beta}
    =
    \sum\limits_{i=0}^{\alpha-1} \frac{ \binom{i + \beta - 1}{i} }{(1 - t x)^{\alpha - i}}
        \frac{(-y/x)^i }{(1 - y/x)^{\beta + i}}
    +
    \sum\limits_{j=0}^{\beta-1} \frac{ \binom{j + \alpha - 1}{j} }{(1 - t y)^{\beta - j}}
        \frac{(-x/y)^j }{(1 - x/y)^{\alpha + j}}.
\end{equation*}
\end{corollary}


\subsection{Description of the Algorithm}
\label{subsec:DirtyMethod}

Let $V$ be a reducible representation of $\operatorname{SL}_2$. For simplicity we will assume that $V$ has
no trivial subrepresentations. The gist of the algorithm is formula \eqref{eq:finalformula} below.

To describe the algorithm, it will be convenient to introduce a new notation
for the decomposition of $V$ into irreducible representations as follows. Decompose the representation into
$V=V_{\operatorname{even}}\oplus V_{\operatorname{odd}}$, where $V_{\operatorname{even}}$
consists of those representations whose irreducible components have even degree and
$V_{\operatorname{odd}}$ consists of those representations whose irreducible components have odd degree.
Let $d_1>d_2>\ldots> d_r>0$ denote the (even) degrees of the irreducible components of $V_{\operatorname{even}}$
and $e_1>e_2>\ldots>e_s>0$ the (odd) degrees of the irreducible components of $V_{\operatorname{odd}}$.
Then we can write
\[
    V_{\operatorname{even}}=\bigoplus_{i=1}^r V_{d_i}^{m_i},
\]
where $m_i$ is the multiplicity of $V_{d_i}$ and
\[
    V_{\operatorname{odd}}=\bigoplus_{j=1}^s V_{e_j}^{n_j},
\]
where $n_j$ is the multiplicity of $V_{e_j}$.

We now determine the weights of the Cartan torus and their corresponding multiplicities. The even weights $d_1-2i$
for $i=0,1,\ldots, d_1$ occur with multiplicity $\mu_i:= \sum_{k:\:|d_1-2i|\leq d_k}m_k$.
Similarly, the odd weights that occur in $V$ are $e_1-2j$ for $j=0,1,\ldots, e_1$ and
occur with multiplicity $\nu_j:= \sum_{\ell:\:|e_1-2j|\leq e_\ell}n_\ell$.

With this  notation we rewrite the Hilbert series in Equation~\eqref{eq:MainIntFirst} as follows:
\[
    \Hilb_{V}(t)=\frac{1}{2\pi \sqrt{-1}}\int_{\Sp^1} \frac{(1-z^2)\:dz}
        {z\prod\limits_{i=0}^{d_1}(1-t z^{d_1-2i})^{\mu_i}\prod\limits_{j=0}^{e_1}(1-t z^{e_1-2j})^{\nu_j}}\:.
\]
We introduce variables $\bs{x} = (x_0,x_1,\ldots,x_{d_1})$ and $\bs{y} = (y_0,y_1,\ldots,y_{e_1})$
corresponding to the even and odd weights of the Cartan torus. Moreover, we introduce  the function
\[
    \Phi(t,\bs x, \bs y)
        :=  \frac{1}{\prod\limits_{i=0}^{d_1}(1-tx_i)^{\mu_i}\prod\limits_{j=0}^{e_1}(1-ty_j)^{\nu_j}}
\]
and the exceptional set $E$ of points $z\in\Sp^1$ such that
$z^{d_1}$, $z^{d_1-2}, \ldots, z^{-d_1}$, $z^{e_1}$, $z^{e_1-2},\ldots, z^{-e_1}$ are not pairwise distinct.
Noting that $E\subset \Sp^1$ is finite and defining $g:\mathbb C\to \mathbb C^{d_1+e_1+2}$,
$g(z):=(z^{d_1}, z^{d_1-2},\ldots, z^{-d_1},z^{e_1}, z^{e_1-2},\ldots, z^{-e_1})$, we can write
\begin{equation}
\label{eq:Phiformula}
    \Hilb_{V}(t)
    =   \frac{1}{2\pi \sqrt{-1}}\int_{\Sp^1\backslash E}(1-z^2)\Phi(t,g(z)) \frac{dz}{z} \: .
\end{equation}

At this point we make use of  the results of Section~\ref{subsec:PartFrac}. Namely, if
$x_0$, $x_1,\ldots,x_{d_1}$, $y_0$, $y_1,\ldots,y_{e_1}$ are pairwise distinct we use the partial
fraction decomposition
\[
    \Phi(t,\bs x, \bs y)=\sum_{i=0}^{d_1}\sum_{j=0}^{\mu_i-1}
        \frac{G_{i,j}(\bs x, \bs y)}{(1-tx_i)^{\mu_i-j}}
        + \sum_{k=0}^{e_1}\sum_{\ell=0}^{\nu_k-1}
        \frac{H_{k,\ell}(\bs x, \bs y)}{(1-ty_k)^{\nu_k-\ell}},
\]
where
\begin{align*}
    G_{i,j}(\bs x, \bs y)
        &=  \frac{1}{j!(-x_i)^j}\frac{d^j}{dt^j}
            \left.\left({\prod\limits_{\substack{l=0\\l\ne i}}^{d_1}(1-t x_l)^{-\mu_l}
            \prod\limits_{m=0}^{e_1}(1-ty_m)^{-\nu_k} }\right)\right|_{t=1/x_i},
        \\
    H_{k,\ell}(\bs x, \bs y)
        &=  \frac{1}{\ell!(-y_k)^\ell}\frac{d^\ell}{dt^\ell}
            \left.\left({\prod\limits_{l=0}^{d_1}(1-t x_l)^{-\mu_l}
            \prod\limits_{\substack{m=0\\m\ne k}}^{e_1}(1-ty_m)^{-\nu_k}}\right)\right|_{t=1/y_k}.
\end{align*}
Substituting this into Equation~\eqref{eq:Phiformula} we find
\begin{align*}
    \Hilb_{V}(t)
        &=  \frac{1}{2\pi \sqrt{-1}}\int_{\Sp^1\backslash E}
            (1-z^2)\left(\sum_{i=0}^{d_1}\sum_{j=0}^{\mu_i-1}
                \frac{G_{i,j}(g(z))}{(1-tz^{d_1-2i})^{\mu_i-j}}
                \right.
                \\ &\quad\quad\quad\quad\quad\quad\quad\quad \left.
                + \sum_{k=0}^{e_1}\sum_{\ell=0}^{\nu_k-1}
                \frac{H_{k,\ell}(g(z))}{(1-tz^{e_1-2k})^{\nu_k-\ell}}\right)\frac{dz}{z} \: .
\end{align*}
Introducing $\Phi_{i,j}(z):=(1-z^2)G_{i,j}(g(z))$ and $\Psi_{k,\ell}(z):=(1-z^2)H_{k,\ell}(g(z))$ and observing that
the integrands have no singularities along $\Sp^1$ yields
\begin{align}
\label{eq:Phi2}
    \Hilb_{V}(t)
        &=\frac{1}{2\pi \sqrt{-1}}
        \sum_{i=0}^{d_1}\sum_{j=0}^{\mu_i-1}
            \int_{\Sp^1}\frac{\Phi_{i,j}(z)}{(1-tz^{d_1-2i})^{\mu_i-j}}\frac{dz}{z}
    \\ \nonumber &\quad\quad
        +\frac{1}{2\pi \sqrt{-1}}
        \sum_{k=0}^{e_1} \sum_{\ell=0}^{\nu_k-1}
            \int_{\Sp^1}\frac{\Psi_{k,\ell}(z)}{(1-tz^{e_1-2k})^{\nu_k-\ell}}\frac{dz}{z} \: .
\end{align}

Now, for a non-negative integer $a$, recall from \cite[Section 4]{HerbigSeaton}, \cite{Springer} the operation
$U_a\co \Q[\![z]\!]\to\Q[\![t]\!]$ that assigns to a formal power series $F(z) = \sum_{i=0}^\infty F_i z^i$
the series
\[
    (U_a F)(t) := \sum\limits_{i=0}^\infty F_{ia} t^i \in \Q[\![t]\!].
\]
By \cite[Lemma 4.1]{HerbigSeaton}, if $F(t)$ is the power series of a rational function, then $(U_a F)(t)$
is as well. Similarly, if $a\ne 0$, then $U_a$ can be described in terms of averaging over $a$th roots of unity:
\[
   (U_a F)(t) = \frac{1}{a}\sum\limits_{\zeta^a=1} F(\zeta \sqrt[a]{t} ).
\]
We have the following.

\begin{proposition}
\label{prop:UaFormula}
For $n\ge 0$, define the differential operator $D_n:=\frac{d^n}{dz^n}\circ (z^n\cdot)$, i.e.
multiplication by $z^n$ followed by $n$-fold differentiation with respect to $z$. If
$F(z) = \sum_{i=0}^\infty F_i z^i$ is convergent on the closed unit disk, then
\begin{align*}
    \frac{1}{2\pi \sqrt{-1}}
        \int_{\Sp^1}\frac{F(z)}{(1-tz^{-a})^m} \frac{dz}{z}
        =   \frac{1}{(m-1)!}\left(D_{m-1}(U_a F)\right)(t).
\end{align*}
\end{proposition}
\begin{proof}
It is easy to see that both sides are equal to $\sum_{i=0}^\infty \binom{m-1+i}{m-1} F_{ia} t^i$.
\end{proof}

Applying Proposition~\ref{prop:UaFormula} to Equation~\eqref{eq:Phi2}, we obtain
\begin{align}
\label{eq:finalformula}
    \Hilb_{V}(t)
        &=  \sum_{i=0}^{\frac{d_1}{2}} \sum_{j=0}^{\mu_i-1}
                \frac{\left (D_{\mu_i-j-1}\left( U_{d_1-2i} \Phi_{i,j}\right)\right)(t)}{(\mu_i-j-1)!}
    \\ \nonumber &\quad\quad
            + \sum_{k=0}^{\frac{e_1-1}{2}}\sum_{\ell=0}^{\nu_k-1}
                \frac{\left (D_{\nu_k-\ell-1}\left( U_{e_1-2k} \Psi_{k,\ell}\right)\right)(t)}{(\nu_k-\ell-1)!}.
\end{align}

To compute each $(U_{d_1-2i} \Phi_{i,j})(t)$ and $(U_{e_1-2k} \Psi_{k,\ell})(t)$, we follow the process described
in \cite[Section 4]{HerbigSeaton}. Specifically, when computing $(U_a F)(t)$ for a rational function $F(z)$ whose
denominator consists of factors of the form $1 - z^b$, each such factor transforms by the rule
\[
    (1 - z^b)   \longmapsto     (1 - t^{\lcm(a,b)/a} )^{\gcd(a,b)}
\]
to yield the denominator of $(U_a F)(t)$. Then we can determine the numerator
$\operatorname{num}((U_a F)(t))$ of $(U_a F)(t)$ via
\[
    \operatorname{num}((U_a F)(t))
    =   (U_a F)(t) \operatorname{denom}((U_a F)(t)).
\]
Writing $\Hilb_V(t) = P(t)/Q(t)$, we have by Kempf's bound \cite[Theorem 4.3]{KempfHRThrm}
that $\deg(P) \leq \deg(Q)$. Hence for each $F$, as any terms in the numerator with degree larger than
that of the denominator will cancel in the complete expression, we need only determine the Taylor expansion
up to $a\deg(Q)$.

This algorithm has been implemented on \emph{Mathematica} and is available from the authors by request.
We have been able to use it to compute the Hilbert series of large irreducible representations on a PC; as
an example, $\Hilb_{V_{50}}(t)$ was computed in 52 hours. It has denominator
$(1 - t^2)^2\prod_{i=3}^{49}(1 - t^i)$ and a numerator of degree $1175$ with largest coefficient
approximately $1.6996\times 10^{52}$.

For a simple reducible example with multiplicities, $V_2\oplus V_3\oplus V_3$ is computed in a few seconds; the Hilbert
series is given by
\begin{align*}
    &\Big(-t^{21} - t^{18} - 3 t^{17} - 4 t^{16} - 5 t^{15} - 8 t^{14} - 7 t^{13}
        - 3 t^{12} - 2 t^{11} + 2 t^{10} + 3 t^9 + 7 t^8
    \\ &\quad
        + 8 t^7 + 5 t^6 + 4 t^5 + 3 t^4 + t^3 + 1\Big)
    /\Big((1-t^2)^2(1-t^3)^2(1-t^4)^3(1-t^5)^2\Big)  .
\end{align*}


\appendix

\section{Exceptional Cases}
\label{ap:Exceptions}

For the sake of completeness, we give the Hilbert series and Laurent coefficients corresponding to the
representations to which some of our computations of the $\gamma_m$ do not apply
in Table~\ref{tab:Exceptions}.
The Hilbert series for these cases are known and can easily be computed directly using the above
methods; similarly, generating invariants are classical and can easily be computed, e.g. using the
algorithms described in \cite{BedratyukSL2Invariants} and \cite[Sections 4.1--2]{DerskenKemperBook}.
Note that our formulas for $\gamma_0$ and $\gamma_1$ apply to all cases except
$V_1$, $2V_1$, $V_2$, $V_3$, and $V_4$; the remaining cases are exceptions to the results given for
$\gamma_2$ and $\gamma_3$. Curiously, $\gamma_0$ is the reciprocal of an integer for all exceptions
listed. This is not the case in general; for instance, $V_7$ has $\gamma_0 = 11/11520$.

\begin{table}[h]
\label{tab:Exceptions}
\begin{tabular}{|c|c|c|c|c|c|c|c|}
\hline
    $\bs{V}$    &   $\bs{D}$&   $\bs{\Hilb_V(t)}$
    &   $\bs{\gamma_0}$     &   $\bs{\gamma_1}$     &   $\bs{\gamma_2}$     &   $\bs{\gamma_3}$
    &   $\bs{a(\C[V]^{\SL_2})}$
\\ \hline\hline
    $V_1$       &   $2$     &   $1$
    &   $1$                 &   $0$                 &   $0$                 &   $0$
    &   $2 - D = 0$
\\ \hline
    $V_2$       &   $3$     &   $\frac{1}{1 - t^2}$
    &   $\frac{1}{2}$       &   $\frac{1}{4}$       &   $\frac{1}{8}$       &   $\frac{1}{16}$
    &   $1 - D = -2$
\\ \hline
    $V_3$       &   $4$     &   $\frac{1}{1 - t^4}$
    &   $\frac{1}{4}$       &   $\frac{3}{8}$       &   $\frac{5}{16}$      &   $\frac{5}{32}$
    &   $-D = -4$
\\ \hline
    $V_4$       &   $5$     &   $\frac{1}{(1 - t^2)(1 - t^3)}$
    &   $\frac{1}{6}$       &   $\frac{1}{4}$       &   $\frac{17}{72}$     &   $\frac{25}{144}$
    &   $-D = -5$
\\ \hline
    $V_5$       &   $6$     &   $\frac{t^{18} + 1}
                                    {(1 - t^4) (1 - t^8) (1 - t^{12})}$
    &   $\frac{1}{192}$     &   $\frac{1}{128}$     &   $\frac{199}{1152}$  &   $\frac{965}{2304}$
    &   $-D = -6$
\\ \hline
    $V_6$       &   $7$     &   $\frac{t^{15} + 1}
                                    {(1 - t^2) (1 - t^4) (1 - t^6) (1 - t^{10})}$
    &   $\frac{1}{240}$     &   $\frac{1}{160}$     &   $\frac{71}{720}$    &   $\frac{17}{72}$
    &   $-D = -7$
\\ \hline
    $V_8$       &   $9$     &   $\frac{t^{18} + t^{10} + t^9 + t^8 + 1}
                                    {\prod_{m=2}^7 (1 - t^m) }$
    &   $\frac{1}{1008}$    &   $\frac{1}{672}$     &   $\frac{191}{15120}$ &   $\frac{11}{378}$
    &   $-D = -9$
\\ \hline
    $2V_1$      &   $4$     &   $\frac{1}{1 - t^2}$
    &   $\frac{1}{2}$       &   $\frac{1}{4}$       &   $\frac{1}{8}$       &   $\frac{1}{16}$
    &   $2-D = -2$
\\ \hline
    $2V_2$      &   $6$     &   $\frac{1}{(1 - t^2)^3}$
    &   $\frac{1}{8}$       &   $\frac{3}{16}$      &   $\frac{3}{16}$      &   $\frac{5}{32}$
    &   $-D = -6$
\\ \hline
    $2V_3$      &   $8$     &   $\frac{t^{10} + t^6 + t^4 + 1}{(1 - t^2) (1 - t^4)^4}$
    &   $\frac{1}{128}$     &   $\frac{3}{256}$     &   $\frac{23}{512}$    &   $\frac{95}{1024}$
    &   $-D = -8$
\\ \hline
    $2V_4$      &   $10$    &   $\frac{t^8 + t^4 + 1}{(1 - t^2)^3 (1 - t^3)^4}$
    &   $\frac{1}{216}$     &   $\frac{1}{144}$     &   $\frac{11}{432}$    &   $\frac{5}{96}$
    &   $-D = -10$
\\ \hline
    $V_1 + V_2$ &   $5$     &   $\frac{1}{(1 - t^2)(1 - t^3)}$
    &   $\frac{1}{6}$       &   $\frac{1}{4}$       &   $\frac{17}{72}$     &   $\frac{25}{144}$
    &   $-D = -5$
\\ \hline
    $V_1 + V_3$ &   $6$     &   $\frac{t^6 + 1}{(1 - t^4)^3}$
    &   $\frac{1}{32}$      &   $\frac{3}{64}$      &   $\frac{9}{64}$      &   $\frac{35}{128}$
    &   $-D = -6$
\\ \hline
    $V_1+V_4$   &   $7$     &   $\frac{t^9 + 1}{(1 - t^2) (1 - t^3) (1 - t^5) (1 - t^6)}$
    &   $\frac{1}{90}$      &   $\frac{1}{60}$      &   $\frac{109}{1080}$  &   $\frac{97}{432}$
    &   $-D = -7$
\\ \hline
    $V_2+V_3$   &   $7$     &   $\frac{t^7 + 1}{(1 - t^2) (1 - t^3) (1 - t^4) (1 - t^5)}$
    &   $\frac{1}{60}$      &   $\frac{1}{40}$      &   $\frac{71}{720}$    &   $\frac{59}{288}$
    &   $-D = -7$
\\ \hline
    $V_2+V_4$   &   $8$     &   $\frac{t^6 + 1}{(1 - t^2)^2 (1 - t^3)^2 (1 - t^4)}$
    &   $\frac{1}{72}$      &   $\frac{1}{48}$      &   $\frac{29}{432}$    &   $\frac{115}{864}$
    &   $-D = -8$
\\ \hline
\end{tabular} \vspace{.3cm}
\caption{The Hilbert series of the representations $V$ that occur as exceptions to part of
Theorem~\ref{thrm:Main}.}
\end{table}


\bibliographystyle{amsplain}
\bibliography{CHHS-SL2}

\end{document}